\documentclass[reqno]{amsart}[11pt,draft,amssymb]
\usepackage{amsmath,amssymb,amscd,amsfonts}
\usepackage[all]{xy}
\usepackage{enumerate}

{\catcode`\@=11
\gdef\n@te#1#2{\leavevmode\vadjust{%
 {\setbox\z@\hbox to\z@{\strut#1}%
  \setbox\z@\hbox{\raise\dp\strutbox\box\z@}\ht\z@=\z@\dp\z@=\z@%
  #2\box\z@}}}
\gdef\leftnote#1{\n@te{\hss#1\quad}{}}
\gdef\rightnote#1{\n@te{\quad\kern-\leftskip#1\hss}{\moveright\hsize}}
\gdef\?{\FN@\qumark}
\gdef\qumark{\ifx\next"\DN@"##1"{\leftnote{\rm##1}}\else
 \DN@{\leftnote{\rm??}}\fi{\rm??}\next@}}
\DeclareFontFamily{OT1}{wncyr}{\hyphenchar\font45 }
\DeclareFontShape{OT1}{wncyr}{m}{n}{%
   <5> <6> <7> <8> <9> gen * wncyr
   <10> <10.95> <12> <14.4> <17.28> <20.74>  <24.88>wncyr10}{}
\DeclareFontShape{OT1}{wncyr}{m}{it}{%
   <5> <6> <7> <8> <9> gen * wncyi
   <10> <10.95> <12> <14.4> <17.28> <20.74> <24.88> wncyi10}{}
\DeclareFontShape{OT1}{wncyr}{m}{sc}{%
   <5> <6> <7> <8> <9> <10> <10.95> <12> <14.4>
   <17.28> <20.74> <24.88>wncysc10}{}
\DeclareFontShape{OT1}{wncyr}{b}{n}{%
   <5> <6> <7> <8> <9> gen * wncyb
   <10> <10.95> <12> <14.4> <17.28> <20.74> <24.88>wncyb10}{}
\input cyracc.def
\def\rus{\usefont{OT1}{wncyr}{m}{n}\cyracc\fontsize{8}{10pt}\selectfont}
\def\rusit{\usefont{OT1}{wncyr}{m}{it}\cyracc\fontsize{8}{10pt}\selectfont}

\DeclareMathSizes{9}{9}{7}{5} % This only is different.

\theoremstyle{plain}

\newtheorem{theorem}{Theorem}[section]

\newtheorem{lemma}[theorem]{Lemma}
\newtheorem{remark}[theorem]
{Remark}
\newtheorem{corollary}[theorem]{Corollary}

\theoremstyle{definition}

\newtheorem{definition}[theorem]{Definition}

\newtheorem{nothing*}[theorem]{}
\newtheorem{subnothing*}[sub]{}
\newtheorem{example}[theorem]{Example}
\newtheorem{examples}[theorem]{Examples}

\newtheorem{claim}{Claim}

\theoremstyle{remark}

\def\cod{{\rm codim}}

\def\GG{{G/\!\!/G}}
\def\pG{{\pi^{}_{G}}}

\def\tG{\widehat{G}}
\def\tGG{{\widehat{G}/\!\!/\widehat{G}}}
\def\ptG{{\pi_{\widehat G}}}
\def\tS{\widehat{S}}
\def\tC{\widehat{C}}
\def\tT{\widehat{T}}
\def\mmu{{\mbox{\boldmath$\mu$}}}

\begin{document}
%%\renewcommand{\baselinestretch}{2}

%%\

%%\vskip -10mm

%%\

\title[Cross-sections, quotients, and representation rings]{Cross-sections,
quotients, and \\
representation rings of\\ semisimple algebraic groups}

\author[Vladimir  L. Popov]{Vladimir  L. Popov${}^*$}
\address{Steklov Mathematical Institute,
Russian Academy of Sciences, Gubkina 8, Moscow\\
119991, Russia} \email{popovvl@orc.ru}

\thanks{
 ${}^*$\,Supported by
 grants {\rus RFFI
08--01--00095}, {\rus N{SH}--1987.2008.1}, and the
program {\it Contemporary Problems of Theoretical
Mathematics} of the Russian Academy of Sciences, Branch
of Mathematics. }

%%\date{August 20, 2009}



\dedicatory{ \hfill{
%%\parbox
\begin{minipage}[t]{2.4in}
{\raggedright {\rm``}Is Steinberg's theorem $[\ldots]$
%%(in the form that says
%%%%\\
%%that rational points exist on rational orbits,
%%\\
%%say)
only true for simply connected groups $[\ldots]$\hskip
.1mm{\rm ?}
%%, and is this also true of the corollary which %%states that the structural group can always be %%reduced to a maximal torus?
What hap\-pens for ${\rm GP}(1)$, for instance\,{\rm?}
Is there a rational section of $\;G$ over $I(G)$
{\rm(\!``}inva\-ri\-ants\,{\rm'')} in this case\,{\rm?}
$[\ldots]$ Is it true
that $I(G)$ is a rational variety $[\ldots]$\,{\rm ?''}\\
\hskip 10mm {\rm A. Grothendieck}, {\it Letter to J.-P.
Serre},\\ \hskip 10mm {\rm January} {\rm 15, 1969}, {\rm
\cite[pp.\,240--241]{GS}.}}
\end{minipage}
} }

\maketitle

%%\renewcommand{\baselinestretch}{1.0}

%%%%%%%%%%%%%%%%%%%%%%%%%%%%%%%%abstract
\begin{abstract}
Let $G$ be a connected semisimple algebraic group over
an algebraically closed field $k$. In 1965 {\sc
Steinberg} proved that if $G$ is simply connected, then
in $G$ there exists a closed irreducible cross-section
of the set of closures of regular conjugacy classes. We
prove that in arbitrary $G$ such a cross-section  exists
if and only if the universal covering isogeny
$\tau\colon \tG\to G$ is bijective; this answers {\sc Grothendieck}'s question
cited in the epigraph.\;In particular, for
${\rm char}\,k=0$, the converse to {\sc Steinberg}'s
theorem holds. The existence of a cross-section in $G$
implies, at least for ${\rm char}\,k=0$, that the
algebra $k[G]^G$ of class functions on $G$ is generated
by ${\rm rk}\,G$ elements.  We describe, for arbitrary
$G$, a minimal ge\-ne\-rating set of
$k[G]^G$ and that of the representation ring of $G$
and answer two {\sc Grothendieck}'s questions on
constructing generating sets of $k[G]^G$.
We prove the existence of a
rational (i.e., local) section of the quotient morphism for arbitrary $G$
and the
existence of a rational cross-section in $G$ (for
${\rm char}\,k=0$, this has been proved earlier);
this answers the other {\sc Grothendieck}'s question
cited in the epigraph.
We also prove that the existence of a rational section is equi\-va\-lent to the existence of a rational
$W$-equi\-va\-riant map $T\dashrightarrow G/T$ where $T$ is
a maximal torus of $G$ and $W$ the Weyl group.
\end{abstract}

\section{Introduction}

Below all algebraic varieties are taken over an
algebraically closed field $k$. We use the standard
notation and conventions of \cite{Bor} and \cite{Sp}.

Let $G$ be a connected semisimple algebraic group,
$G\neq \{e\}$. Let $(\GG, \pG)$ be a categorical
quotient for the conjugating action of $G$ on itself,
i.e., $\GG$  is an affine variety (quotient variety) and
%%\ \vskip -5mm
\vskip 1mm
\begin{equation}\label{piG}
\pG\colon G\longrightarrow \GG
\end{equation}
\vskip 2mm
\noindent
a surjective  morphism (quotient morphism) such that
$\pi^*_G\big(k[\GG]\big)$ is the algebra $k[G]^G$ of
class functions on $G$. Every fiber of $\pG$ is then the
closure of a regular conjugacy class  (i.e., that of the
maximal dimension) and such classes in general position
are closed  \cite[Theorem 6.11, Cor. 6.13, and Sect.
2.14]{St1}.

\begin{definition}
A closed irreducible subvariety $S$ of $G$ is called a
{\it cross-section} (of the collection of fibers of
$\pG$) in $G$ if $S$ intersects
every fiber of $\pG$  at a single point.
\end{definition}

The elements of $S$
are the ``canonical forms'' of the elements of a dense
constructible subset of $G$ with respect to conjugation.\;The image of any {\it section}
 of $\pG$ (i.e., a morphism $\sigma\colon \GG\to G$ such that $\pG\circ\sigma={\rm id}^{}_{\GG}$)
is an example of such $S$; moreover, this $S$ has the property that $\pG|_S\colon S\to \GG$ is an isomorphism. For ${\rm char}\,k=0$,
every cross-section in $G$ is obtained in this manner
(see Subsection \ref{re}.A).
%% in Section \ref{re}).

In 1965 {\sc Steinberg} gave an explicit construction of
a section of $\pG$ for every simply connected semisimple
group $G$ (see his celebrated paper \cite{St1}).
Its image is a cross-section that intersects every
regular conjugacy class
 and does not intersect other conjugacy classes.

In this paper we explore what happens in the general
case, i.e., when  $G$ is not necessarily simply
connected. In this case   the following two facts about
cross-sections in $G$ for ${\rm char}\,k=0$ are known.

First, by \cite[Theorem 0.3]{CTKPR} in every connected
semisimple algebraic group $G$ there is a {\it rational
section} of $\pG$, i.e., a section over
a dense open subset of $\GG$ (local section).

Second, by {\sc Kostant}'s theorem \cite[Theorem
0.10]{K} there is an infinitesimal counterpart of {\sc
Steinberg}'s cross-section:
 for
 the adjoint action of $G$ on its Lie algebra
${\rm Lie}\,G$, there is a closed irreducible subvariety
in ${\rm Lie}\,G$ that intersects   every regular
$G$-orbit at a single point.

In order to formulate our result consider the universal
covering of $G$, i.e.,
%%an
a central
isogeny
\vskip -1mm
\begin{equation*}
\tau\colon \tG\longrightarrow G
\end{equation*}
\vskip 2mm
\noindent
such that $\tG$ is a simply connected semisimple
algebraic group (by \cite[Prop. (2.24)(ii)]{BT}
%%such $\tau$
it exists and unique up to
$G$-isomorphism).
%% and the composition of $\tau$ with every
%%projective rational representation of $G$ lifts to a
%%linear one of $\tG$.

We prove
the following

\begin{theorem}\label{main} Let $G$ be a connected semisimple algebraic group.
\begin{enumerate}
\item[\rm(i)] The following properties are equivalent:
\begin{enumerate}
\item[\rm(a)] there is a cross-section in $G$;
\item[\rm(b)] the isogeny $\tau$ is bijective.
\end{enumerate}
\item[\rm(ii)] If  $\sigma\colon \GG\to G$ is a section
of $\pG$, then the cross-section $\sigma(\GG)$ in $G$
in\-ter\-sects every regular conjugacy class and does
not intersect other conjugacy classes.
\end{enumerate}
\end{theorem}

\begin{remark} {\rm
The isogeny $\tau$ is bijective if and only if it is
either an isomorphism or purely inseparable (radical).
The latter holds if and only if ${\rm char}\,k=p>0$ and
$p$ divides the order of the fundamental group of $G$.}
\end{remark}

%%Statement (ii) of the
The next corollary answers the first
{\sc Grothendieck}'s question
cited in the epigraph
and the question
posed in \cite[p.\,4]{CTKPR}.

\begin{corollary} Let $G$ be a connected semisimple algebraic group.

\begin{enumerate}
\item[\rm(i)] If a section of $\pG$ exists, then $\tau$
is bijective. \item[\rm(ii)] For ${\rm char}\,k=0$, the
following properties are equivalent:
\begin{enumerate}
\item[\rm(a)] there is a section of $\pG$; \item[\rm
(b)] there is a cross-section in $G$; \item[\rm (c)] $G$
is simply connected.
\end{enumerate}
\end{enumerate}
\end{corollary}

Theorem \ref{main} is proved in Section \ref{cs}.

One can show (see below Lemma \ref{smooth}) that if a
cross-section in $G$ exists, then, at least for ${\rm
char}\,k=0$, the variety $G/\!\!/G$ is smooth (the
converse is not true). The known criterion of smoothness
of $\GG$ (Theorem \ref{sing}) may be
interpreted as that of the existence of ${\rm rk}\,G$
generators of $k[G]^G$.
 In Section~\ref{singugen}  we consider the general case and describe
a minimal generating set of $k[G]^G$ and singularities
of $\GG$ for any $G$. This is based on the property that
actually $\GG$ is a toric variety of a maximal torus $T$
of $G$. In particular, it also implies
the affirmative answer to the last {\sc Grothendieck}'s question
cited in the epigraph:

\vskip 2mm

\noindent{\bf Corollary \ref{ratio}.} {\it $\GG$ and
$T/W$ are rational varieties.}

\vskip 2mm

\noindent Here $W=N_G(T)/T$, where
$N_G(T)$ is the normilizer of $T$ in $G$,
 is the Weyl group of $G$.
 %%, i.e., the
%%quotient of $T$ in its normalizer $N_G(T)$, acting
It acts on
$T$ via conjugation.

Parallel to this we describe a minimal generating set of
the representation ring $R(G)$ of $G$. Note that finding
generators of $R(G)$ attracted people's attention during
long time, in particular, because of the bearing on the
$K$-theory (cf., e.g., \cite[Chap.\;13]{Hus}  where the
generators of $R(G)$ are found for some classical $G$'s
utilizing the ad hoc bulky arguments; see
also~\cite{A}). Singularities of $\GG$ attracted people's
attention as well
(see \cite[Sects.\;3.15, 4.5]{Sl}).

 The precise
formulations of these results
are given below in Theorems \ref{toric} and
\ref{generators} and Lemma \ref{explcenter}.

Constructing generating sets of $k[G]^G$ is the topic
of two further questions of {\sc Gro\-thendieck} asked
in \cite[p.\;241]{GS}. In Section \ref{Yettwo} we
answer the first question
in the ne\-ga\-tive and the second
in the positive.

In Section \ref{rcs} we consider rational (i.e.,\;local) sections of
$\pG$ and {\it rational cross-sections} in $G$, i.e.,
irreducible closed subsets $S$ of $G$ that intersect at
a single point every fiber of $\pG$ over a point of a
dense open subset of $\GG$ (depending on $S$). The closure of the image of
a rational section of $\pG$ is an example of such $S$; moreover, this $S$ has the property that $\pG|_S\colon S\to \GG$ is a birational isomorphism.
For ${\rm char}\,k=0$, every rational
cross-section in $G$ is obtained in this way.

First, we show that
the existence of a rational section of\,$\pG$ is
equivalent to
another property. Namely, note that
$W$ also acts on $G/T$ as follows:\vskip -1mm
\begin{equation}\label{G/T}
w\cdot gT:=g\overset{.}{w}^{-1}T,
\end{equation}\vskip 2mm\noindent where $\overset{.}{w}\in N_G(T)$ is a representative of $w$.

We prove

\begin{theorem}\label{rsrcs} Let $G$ be a connected semisimple algebraic group. The following pro\-perties are equivalent:
\begin{enumerate}
\item[\rm(i)] there is a rational section of $\;\pG$;
\item[\rm(ii)] there is a $W$-equivariant rational map
$T\dashrightarrow G/T$.
\end{enumerate}
\end{theorem}

Then we consider the existence problem and prove the
following.

First, the next theorem answers the third {\sc Grothendieck}'s
question cited in the epigraph.

\begin{theorem}\label{anychar}  For every connected semisimple algebraic group
$G$, there is a rational section of $\pG$.
%%\begin{enumerate}
%%\item[\rm(i)] There is a rational section of $\pG$.
%%\item[\rm(ii)] There is a rational cross-section $S$ in $G$ such that %%$\pG|_S\colon S\to\GG$ is a birational isomorphism.
%%\end{enumerate}
\end{theorem}

For
${\rm char}\,k=0$, this theorem has been proved earlier
in
\cite[Theorem 0.3]{CTKPR}.
In our proof we use the relevant characteristic free
results from \cite{CTKPR}, but bypass Theorem 2.12 from
this paper
(whose proof is based on the assumption ${\rm
char}\,k=0$) by exploring properties of $\pG$ and
proving that versality of $G$ holds in arbitrary
characteristic (Lemma \ref{G/Tversal}); this permits us to use {\sc Steinberg}'s
section of $\ptG$ in place of {\sc Kostant}'s
cross-section in ${\rm Lie}\,G$ used in \cite{CTKPR}.

\begin{corollary}\label{rtcrsct} In every connected semisimple algebraic group
$G$ there is a rational cross-section $S$ such that $\pG|_S\colon S\to \GG$ is a birational isomorphism.
\end{corollary}

Second, Theorems \ref{rsrcs} and \ref{anychar} yield the
following

\begin{theorem}\label{ratcs}
For every connected semisimple algebraic group $G$,  there is a
$W$-equivariant rational map $\,T\dashrightarrow G/T$.
\end{theorem}

Section \ref{re} contains some remarks, questions, and
an example of a cross-section $S$ in $G$ such that
$\pG|_S$ is not separable (hence $S$ is not the image of
a section of~$\pG$).

\vskip 2mm

{\it Acknowledgements.} I am grateful to {\sc
Vik.\;Kulikov} for the inspiring remark, to {\sc A. Parshin} for the useful discussion and to {\sc J.-L.
Colliot-Th\'el\`ene}, {\sc G. Prasad}, and {\sc Z. Rei\-ch\-stein} for the
valuable comments. I am also indebted to the referee for
%%very
careful reading and thoughtful suggestions on the exposition.
%%for the
%%excellent job.
%%My thanks also go to {\sc G. Prasad} who
%%drew my attention to some important results in \cite{BT}.
%% that lead to
%%strengthening of the formulations of Theorem \ref{anychar} and Corollary \ref{ratcs}.

\section{Cross-sections in \boldmath$G$}\label{cs}\label{crssect}

Given a torus $S$, below we denote by ${\rm X}(S)$ the character lattice of $S$ in additive notation. To avoid confusion between the additions in ${\rm X}(S)$ and $k[S]$, an element
$\lambda\in {\rm X}(S)$ considered as that of $k[S]$ is denoted by $\chi^\lambda$.
%%Thus,
%%\begin{equation}
%%\chi^\mu\chi^\nu=\chi^{\mu+\nu}\qquad
%%\end{equation}
The value of
$\chi^\lambda$ at $s\in S$ is denoted by $s^\lambda$.

Fix a choice of Borel subgroup $\widehat B$ of $\tG$ and
maximal torus $\widehat T\subset \widehat B$.
%%Denote by
%%${\rm X}(\widehat T)$ the
%%character lattice
%%of $\widehat T$ in additive notation. For $\lambda\in
%%{\rm X}(\widehat T)$, denote by $t^\lambda$ the value of
%%$\lambda\colon \widehat T\to {\bf G}_m$ at $t\in
%%\widehat T$.
Let $$\varpi_1,\ldots,\varpi_r\in
{\rm X}(\widehat T)$$ be the system of fundamental weights of
$\widehat T$ regarding~$\widehat B$.

Let $\varrho_i\colon \tG\to {\bf GL}(V_i)$ be an
irreducible representation of $\tG$ with $\varpi_i$ as
the highest weight. Let $%%\chi_{\varpi_i}
{\rm ch}_{\varpi_i}\in
k[\tG]^{\tG}$ be the character of $\varrho_i$.

Let $\widehat C$ be the center of $\tG$; it is a finite
subgroup of $\widehat T$. The conjugating action of
$\tG$ on itself commutes with the action of  $\widehat
C$ on $\tG$ by left translations. Therefore the latter
action  descends to  $\tGG$ and
%%\ \vskip -4mm
$$\ptG\colon \tG\longrightarrow \tG/\!\!/\tG$$
becomes a $\widehat C$-equivariant morphism.

Endow the $r$-dimensional affine space ${\mathbf A}\!^r$
with the linear action of $\tT$ by the formula
\begin{equation}\label{action}
t\cdot(a_1,\ldots,a_r):=(t^{\varpi_1}a_1,
\ldots,t^{\varpi_r}a_r),\qquad t\in \tT, \hskip 2mm
(a_1,\ldots, a_r)\in {\mathbf A}\!^r.
\end{equation}

\begin{lemma}\label{lin}
\

\begin{enumerate}
\item[\rm(i)] The $\tT$-stabilizer of the point
$(1,\ldots,1)\in {\bf A}^r$ is trivial. In particular,
the considered action of $\;\tT$ on ${\mathbf A}\!^r$ is
faithful. \item[\rm(ii)] There is a $\widehat
C$-equivariant isomorphism
$$\lambda\colon \tGG\overset{\simeq}\longrightarrow  {\mathbf A}\!^r.$$
\end{enumerate}
\end{lemma}
\begin{proof}
 Since $\varpi_1,\ldots,\varpi_r$
 generate ${\rm X}(\widehat T)$, we have
 %%\ \vskip -4mm
\begin{equation}\label{kerpi}
\bigcap_{i=1}^{r} \{t\in T\mid t^{\varpi_i}=1\}=\{e\}.
\end{equation}
But \eqref{action} entails that the $\tT$-stabilizer of
the point $(1,\ldots,1)$ coincides with  the
%%right-hand
left-hand
side of equality \eqref{kerpi}.
%%This
Hence (i) follows from this equality.
%%proves (i).

 By \cite[Theorems 6.1, 6.16]{St1}  the $k$-algebra $k[\tG]^{\tG}$ is
 freely generated by $
 %%\chi_{\varpi_1}
 {\rm ch}_{\varpi_1},\ldots$
 $\ldots,
 %%\chi_{\varpi_r}
 {\rm ch}_{\varpi_r}$ and the morphism
 %%\ \vskip -4mm
\begin{equation*}
\theta\colon\tG\longrightarrow {\bf A}\!^r,\qquad
\theta(g)=(%%\chi_{\varpi_1}
{\rm ch}_{\varpi_1}(g),\ldots,
%%\chi_{\varpi_r}
{\rm ch}_{\varpi_r}(g)),
\end{equation*}
is surjective. Hence there is
an isomorphism
$\lambda\colon \tGG\longrightarrow {\bf A}\!^r $
such that the following diagram is commutative: \ \vskip
-5mm
\begin{equation}\label{commu}
\begin{matrix}
\xymatrix@C=6mm@R=5mm{&\widehat G\ar[dl]_{\ptG}\ar[dr]^\theta&\\
\tGG\ar[rr]^\lambda&&{\bf A}\!^r}
\end{matrix}\quad .
\end{equation}

The morphism $\theta$ is $\widehat C$-equivariant.
Indeed, let $c\in \widehat C$. Since $\varrho_i$ is
irreducible, {\sc Schur}'s lemma entails that
$\varrho_i(c)=\mu_{i, c}\,{\rm id}_{V_i}$
for some
$\mu_{i, c}\in k$. On the other hand, since
$c\in \widehat T$, any highest vector in $V_i$
regarding  $\widehat B$ is an eigenvector
of $c$ with the eigenvalue $c^{\varpi_i}$.
Hence $\mu_{i, c}= c^{\varpi_i}$.
 Therefore, for every
$g\in\tG$, by \eqref{action} we have
\begin{align*}
\theta(cg)&=\big(%%\chi_{\varpi_1}
{\rm ch}_{\varpi_1}(cg),\ldots,
%%\chi_{\varpi_r}
{\rm ch}_{\varpi_r}(cg)\big)\\
&=\big({\rm trace}\,\big(\varrho_1(cg)\big),
\ldots,{\rm trace}\,\big(\varrho_r(cg)\big)\big)\\
%%\end{align*}
%%\begin{align*}
&=\big({\rm trace}\,\big(\varrho_1(c)\varrho_1(g)\big),
\ldots,{\rm
trace}\,\big(\varrho_1(c)\varrho_r(g)\big)\big)\\
&=\big({\rm
trace}\,\big(c^{\varpi_1}\!\varrho_1(g)\big),
\ldots,{\rm trace}\,\big(c^{\varpi_r}\!\varrho_r(g)\big)\big)\\
&=\big(c^{\varpi_1}{\rm
trace}\,\big(\varrho_1(g)\big),\ldots, c^{\varpi_r}{\rm
trace}\,\big(\varrho_r(g)\big)\big)\\
&=\big(c^{\varpi_1}%%\chi_{\varpi_1}
{\rm ch}_{\varpi_1}(g),\ldots, c^{\varpi_r}
%%\chi_{\varpi_r}
{\rm ch}_{\varpi_r}(g)\big)\\
&=c\cdot \theta(g),
\end{align*}
as claimed.

   Since both $\theta$ and $\ptG$ are $\widehat C$-equivariant and $\ptG$ is surjective,
commutativity of diagram \eqref{commu} entails that
$\lambda$ is $\widehat C$-equivariant as well. This
proves (ii). \quad $\square$
\renewcommand{\qed}{}\end{proof}

\begin{corollary}\label{no}
 Let  $g$ be a nonidentity element of $\,\widehat C$.
Then there is no $g$-stable cross-section in $\tG$.
\end{corollary}
\begin{proof}  Assume the contrary and let $\tS$ be a %%$Z$
$g$-stable cross-section in $\tG$. Since $\ptG$ is
$\widehat C$-equivariant, $\ptG|_{\tS}\colon \tS\to \tGG
$
is a bijective
$g$-equivariant morphism.
As, by Lemma \ref{lin}(ii),
there is a point of $\tGG$ fixed by $\widehat C$, hence
by $g$,
this
implies that there is a point of
$\widehat S$
fixed by $g$. But  for the action of $\widehat C$ on
$\tG$ by left translations, the stabilizer of every
point is trivial, a contradiction with $g\neq e$.
\quad $\square$
\renewcommand{\qed}{}\end{proof}

Given an element $h$ of an algebraic group $H$, we shall
denote its conjugacy class in $H$ by $H(h)$:\vskip -1mm
\begin{equation}\label{cc}
H(h):=\{shs^{-1}\mid s \in H\}.
\end{equation}
\vskip 2mm

\begin{lemma}\label{prop}
Let $H$ and $\widetilde H$ be connected algebraic groups
and let $\sigma\colon \widetilde H\to H$ be an isogeny.
Then the following properties hold:
\begin{enumerate}
\item[\rm(i)] $\sigma$ is a finite morphism;
\item[\rm(ii)] $\sigma\big(\widetilde
H(h)\big)=H\big(\sigma(h)\big)$ and $\dim \widetilde
H(h)=\dim H\big(\sigma(h)\big)$
for every $h\in \widetilde H$;
\item[\rm(iii)] if $\,\widetilde H(h)$ is a regular
conjugacy class in $\widetilde H$ {\rm(}i.e., that of
the maximal dimension{\,\rm)}, then
$\,\sigma\big(\widetilde H(h)\big)$ is a regular
conjugacy class in $H$; \item[\rm(iv)] if $H$ and
$\widetilde H$ are semisimple, then for every $h\in
\widetilde H$, $$\sigma\big(\pi^{-1}_{\widetilde
H}\big(\pi^{}_{\widetilde H}(h)\big)\big)
=\pi^{-1}_{H}\big(\pi^{}_{H}\big(\sigma(h)\big)\big).$$
\end{enumerate}
\end{lemma}

\begin{proof}
The varieties $H$ and $\widetilde H$ are normal (even
smooth) and  the fiber of $\sigma$ over every point of
$\,H$ is a finite set whose cardinality does not depend
on this point.
Hence (cf.\;\cite[Sect.\,2, Cor.\,3]{G}) $\widetilde H$
is the normalization of $\,H$ in the field of rational
functions on $\widetilde H$ and $\sigma$ is the
normalization map. This proves (i).

The first equality in (ii) holds as $\sigma$ is an
epimorphism of groups. The second follows from the
first and theorem on dimension of fibers
\cite[AG\,10.1]{Bor}. This proves~(ii).

As $\sigma$ is surjective, (iii) follows from (ii).

Since the fibers of $\pi^{}_{\widetilde H}$ and
$\pi^{}_{H}$ are the closures of regular conjugacy
classes and, by (i), the map $\sigma$ is closed, (iv)
follows from (iii).
\quad $\square$
\renewcommand{\qed}{}\end{proof}

\begin{corollary}\label{insep}
Let $\widetilde G$ be a connected semisimple algebraic
group and let $\sigma\colon \widetilde G\to G$ be a
bijective isogeny.
\begin{enumerate}
\item[\rm(i)] If $\,\widetilde S$ is a cross-section in
$\widetilde G$, then $\sigma (\widetilde S)$ is a
cross-section in $G$. \item[\rm(ii)] If $\,S$ is a
cross-section in $G$, then $\sigma^{-1}(S)$ is a
cross-section in $\widetilde G$.
\end{enumerate}
The same holds
if {\rm``}\!cross-section{\rm''} is replaced with
{\rm``}rational cross-section{\rm''}.
\end{corollary}
\begin{proof} By Lemma \ref{prop}(i) the bijective map $\sigma$ is closed. Hence
it is a homeomorphism. Both claims follow from this, the
definitions of cross-section and rational cross-section,
and Lemma \ref{prop}(iv). \quad $\square$
\renewcommand{\qed}{}\end{proof}

\begin{lemma}\label{qZ} Assume that there is a subgroup $Z$ of $\,\widehat C$ such that $G=\tG/Z$ and $\tau$ is the quotient morphism $\tG\to \tG/Z$. Then there is a morphism\vskip 1mm
\begin{equation}\label{phi}
\varphi\colon \tGG\longrightarrow \GG\end{equation}\vskip 2mm\noindent such
that
\begin{enumerate}
\item[\rm(i)] $(\GG, \varphi)$ is
a categorical quotient
for the action of $Z$\,on $\tGG$;
\item[\rm(ii)] the following diagram is commutative:
\begin{equation}\label{dia}
\begin{matrix}
\xymatrix{\tG\ar[r]^\tau\ar[d]_{\ptG}&G\ar[d]^\pG\\
\tGG\ar[r]^\varphi&\GG}
\end{matrix}\quad ;
\end{equation}
\item[\rm(iii)] for every point $x\in \tGG$, the
following equality holds:\vskip 1mm
\begin{equation}\label{fibers}
\tau\big(\pi^{-1}_{\tG}(x)\big)=
\pi^{-1}_{G}\big(\varphi(x)\big).
\end{equation}
\vskip 2mm
\end{enumerate}
\end{lemma}
\begin{proof} As $\tau^*$, $\pi_{\tG}^*$, and $\pi_G^*$
are injections, there is a unique morphism
\eqref{phi} such that $\tau^*\circ\pi_G^*=
\pi_{\tG}^*\circ\varphi^*$, i.e., diagram \eqref{dia} is
commutative.

Consider the action of  $\tG$  on $G$ via the isogeny
$\tau$ and the conjugating action of $G$ on itself. The
isogeny $\tau$ is then $\tG$-equivariant and
$\tG$-orbits in $G$ are $G$-conjugacy classes, so we
have $k[G]^G=k[G]^{\tG}$.  Since the conjugating action
of $\tG$ on itself commutes with the action of $Z$ by
left translations, we have\vskip 1mm
\begin{align*} \pi_{\tG}^*\big(\varphi^*(k[\GG])\big)&=
\tau^*\big(\pi_G^*(k[\GG])\big)=\tau^*\big(k[G]^G\big)=
\tau^*\big(k[G]^{\tG}\big)=
\big(\tau^*(k[G])\big){}^{\tG}\\&=
\big(k[\tG]^Z\big){}^{\tG}=\big(k[\tG]^{\tG}\big){}^Z
=\big(\pi_{\tG}^*(k[\tGG])\big){}^Z=
\pi_{\tG}^*\big(k[\tGG]^Z\big).
\end{align*}
\vskip 3mm\noindent
Thus, $\varphi^*(k[\GG])=k[\tGG]^Z$. This proves (i) and
(ii). Lemma \ref{prop}(iv) and commutativity of diagram
\eqref{dia}
%%yield
imply (iii).\quad $\square$
\renewcommand{\qed}{}\end{proof}

Below, given a variety $Y$, we denote by ${\rm T}_{y,
Y}$ the tangent space of $Y$ at a point~$y$.

\begin{proof}[Proof of Theorem {\rm \ref{main}}]
First, we shall prove criterion (i).

\vskip 1mm ${\it Step}\;1.$ By {\sc Steinberg}'s theorem,
$\tG$ has a cross-section. Hence, by Corollary
\ref{insep}, if
$\tau$
is bijective, then there exists a cross-section in $G$
as well.

So we may assume that
$\tau$ is not bijective and we then have to prove that
there
is no cross-section in $G$. Solving this problem, we may assume that
$\tau$ is separable. Indeed, if this is not the case,
then by  \cite[Prop.\,17.9]{Bor} there exist a connected
semisimple algebraic group $\widetilde G$ and a
commutative diagram of isogenies
\begin{equation}\label{triang}
\begin{matrix}
\xymatrix@C=6mm@R=5mm{\widehat G\ar[rr]^\tau\ar[dr]_\mu&&G\\
&\widetilde G\ar[ur]_\sigma&}
\end{matrix}\quad ,
\end{equation}
where
$\mu$ is
separable
and $\sigma $ is
purely inseparable.
As $\sigma$ is bijective, Corollary \ref{insep} then
reduces the problem  to proving that there is no
cross-section in $\widetilde G$, i.e., we may replace
$G$ by $\widetilde G$ and $\tau$ by $\mu$.

So from now on we may (and shall) assume that
$\tau$ is a separable isogeny of degree $\geqslant 2$.
This means that  there is a nontrivial subgroup $Z$ of
$\widehat C$  such that $G=\tG/Z$ and $\tau$ is the
quotient morphism $\tG\to \tG/Z$.

\vskip 2mm

${\it Step}\;2.$ Now, arguing on the contrary, assume  that
there is a cross-section $S$ in~$G$.

\begin{claim}\

\begin{enumerate}
\item[\rm (i)] {\it For every point $x\in \tGG$, the intersection
\begin{equation}\label{inters}
\pi^{-1}_{\tG}(x) \cap \tau^{-1}(S)\end{equation}
%%}
%%\begin{enumerate}
%%\item[] {\it
is a nonempty subset of a single $Z$-orbit;
in particular, it is finite.}
\item[\rm (ii)] {\it There is a nonempty open subset $\,U$\,of $\,\tGG$ such that,  for every $x\in U$,
intersection \eqref{inters} is a single point.}
\end{enumerate}
\end{claim}

\noindent {\it Proof of Claim $1$.} Consider diagram
\eqref{dia}. Since $S\cap
\pi^{-1}_{G}\big(\varphi(x)\big)$ is a single point $g$,
we deduce from \eqref{fibers}
that
intersection \eqref{inters}
is contained in
$\tau^{-1}(g)$. This proves (i) as the fibers of $\tau$
are $Z$-orbits.

By Lemma \ref{lin}(i) there is a nonempty open subset
$U$ in $\tGG$ such that  the $\widehat C$-stabilizer of
every point of $\,U$ is trivial. Take a point $x\in U$.
Assume that
intersection \eqref{inters} contains two points $g_1$
and $g_2\neq g_1$. By (i) there exists an element $z\in
Z$ such that  $g_2=zg_1$. As $\ptG$ is $\widehat
C$-equivariant,
$x=\ptG(g_2)=\ptG(zg_1)=z\cdot\ptG(g_1)=z\cdot x$. Thus,
$z$ belongs to the $\widehat C$-stabilizer of $x$. The
definition of $\,U$ then implies that $z=e$. Hence
$g_1=g_2$, a contradiction. This proves (ii). \quad
$\square$

\vskip 2mm

${\it Step}\;3.$ Since all the fibers of $\tau$ are finite,
every irreducible component of $\tau^{-1}(S)$ has
dimension $\leqslant \dim\,S=r$ and at least one of them
has dimension $r$.

\begin{claim}
\
\begin{enumerate}
\item[\rm(i)] {\it There is a unique $r$-dimensional irreducible
component $\tS$ of $\,\tau^{-1}(S)$.}
%%\begin{enumerate}
\item[\rm (ii)]
 $\tau(\tS)=S$.
\end{enumerate}
\end{claim}

\noindent{\it Proof of Claim $2$.} Let $\tS$ be an
$r$-dimensional irreducible component of
$\,\tau^{-1}(S)$.
Then $\tau(\tS)$ contains an open subset of $S$. Since
$\tau$ is closed, this proves (ii).

From (ii) we conclude that
\begin{equation}\label{surrr} \pG\big(\tau(\widehat
S)\big)=\GG.
\end{equation} But by Lemma \ref{qZ} the fibers of $\varphi$ in commutative diagram \eqref{dia} are finite. This and \eqref{surrr} imply that $\ptG(\tS)$ contains a nonempty open subset of $\tGG$.

Let now $\tS'$ be another $r$-dimensional irreducible
components of $\,\tau^{-1}(S)$. Then, as above,
$\ptG(\tS')$ contains a nonempty open subset of $\tGG$
as well. Therefore, $\ptG(\tS)\cap\ptG(\tS')$ contains a
nonempty open subset $V$ of $\tGG$. We may assume that
$V\subseteq U$ for $U$ from Claim 1(ii). The latter then
yields that $\pi^{-1}_{\tG}(V)\cap
\tS=\pi^{-1}_{\tG}(V)\cap \tS'.$ As
both sides of this equality are open subsets of
respectively  $\tS$ and $\tS'$, we infer that $\tS=\tS'$.
This proves (i). \quad $\square$

\vskip 2mm

${\it Step}\;4.$  As $\tS$ is a unique $r$-dimensional
irreducible component of the $Z$-stable variety
$\tau^{-1}(S)$, we conclude that $\tS$ is $Z$-stable.
We shall now show that $\tS$ is a cross-section in
$\tG$. As this property contradicts Corollary \ref{no},
the proof of
(i) will be then completed.

\vskip 2mm

${\it Step}\;5.$ Let $x$ be a point of $\tGG$. As $S$ is a
section of $G$, the intersection $S\cap
\pi^{-1}_{G}\big(\varphi(x)\big)$ is a single point
$g\in G$. By Claim 2(ii) there is a point $\widehat
g\in\tS$ such that $\tau(\widehat g)=g$. Commutativity
of diagram \eqref{dia} then entails that $x$ and
$\widehat x:=\ptG(\widehat g)$ are in the same fiber of
$\varphi$. Since the fibers of $\varphi$ are
$Z$-orbits, there is an element $z\in Z$ such that
$x=z\cdot \widehat  x$. As $\ptG$ is $Z$-equivariant,
this yields the equality $\ptG(z\widehat g)=x$. But $z\widehat
g\in\tS$ as $\tS$ is $Z$-stable and $\widehat  g\in
\tS$. Hence $\pi^{-1}_{\tG}(x)\cap \tS\neq
\varnothing$, i.e.,\vskip -1mm
\begin{equation}\label{sur}
\ptG(\tS)=\tGG.
\end{equation}\vskip 2mm

${\it Step}\; 6.$
It follows from Claim 1(i),(ii) and \eqref{sur}
that
$\ptG|_{\tS}$ is a surjective morphism with finite
fibers, bijective over an open subset of $\tGG$. As
$\tG$ is normal, $\tGG$ is normal as well. Let
$\nu\colon \widetilde S\to \tS$ be the normalization.
Then the surjective morphism $\ptG|_{\tS}\circ\nu\colon
\widetilde S\to \tGG$ of normal varieties has finite
fibers and is bijective over an open subset of $\tGG$.
 Hence  $\ptG|_{\tS}\circ\nu$ is bijective
 (see\;\cite[Sect.\,2, Cor.\,2]{G}). Whence $\ptG|_{\tS}$ is bijective as well, i.e., $\tS$ is a cross-section in $\tG$. This completes the proof of (i).

\vskip 1mm

We now turn to the proof of (ii).

Let $S:=\sigma(\GG)$.
Take a point $x\in S$ and put $y:=\pG(x)$.
As $\pG|_S\colon S\to\GG$ is an isomorphism ($\sigma$ is
its inverse),
$d(\pG|_S)_x$ is an isomorphism as well. Hence
$(d\pG)_x$ is surjective. As $\dim {\rm T}_{y,
\GG}\geqslant \dim \GG=r$, this implies that there are
functions $f_1,\ldots, f_r\in k[G]^G$ such that
$(df_1)_x,\ldots, (df_r)_x$ are linearly independent. By
\cite[Theorem 8.7]{St1} this yields that $x$ is regular.
As $S$ intersects every fiber of $\pG$ at a single point
and every such fiber contains a unique regular orbit,
this proves (ii). Thus, the proof of Theorem \ref{main}
comes to a close.
\quad $\square$
\renewcommand{\qed}{}\end{proof}

\section{
Singularities of  \boldmath$\GG$ and generators of $k[G]^G$ and
$R(G)$ }\label{si}\label{singugen}

%%For us, the role of
%%
The
%%the
following statement,
%%lemma,
%%(
whose role for us is
%%only
solely
heuristic,
%%)
%%is heuristic. It
shows that there is a
%%link
connection between
the existence of a cross-section in $G$ and smoothness
of $\GG$.

\begin{lemma}\label{smooth} Let ${\rm char}\,k=0$. If a surjective morphism $\alpha\colon X\to Y$ of irreducible va\-ri\-eties admits a section $\sigma\colon Y\to X$, then smoothness of $X$ implies smoothness of~$\;Y$.
\end{lemma}

\begin{proof}
Arguing on the contrary, assume that $y$ is a singular
point of $Y$, i.e.,\vskip -1mm
\begin{equation}\label{>} \dim {\rm T}_{y, Y}>\dim
Y.
\end{equation}\vskip 2mm

Put $x=\sigma(y)\in X$. Since $\alpha\circ\sigma={\rm
id}_Y$, the composition $d\alpha_x\circ d\sigma_y$ is
the identity map of ${\rm T}_{y, Y}$. Hence $d\alpha_x$
is surjective, i.e., ${\rm rk}\,d\alpha_x=\dim {\rm
T}_{y, Y}$. By \eqref{>} this yields the inequality\vskip -1mm
\begin{equation}\label{rk}
{\rm rk}\,d\alpha_x>\dim Y.
\end{equation}\vskip 2mm

As ${\rm char}\,k=0$, there is a dense open subset $U$
of $X$ such that ${\rm rk}\,d\alpha_z=\dim Y$ for every
point $z\in U$, see \cite[14.4]{Har}. As $z\mapsto \dim
{\rm ker}\,d\alpha_z$ is the upper semi-continuous
function \cite[14.6]{Har}, we conclude that smoothness
of $X$ implies that ${\rm rk}\,d\alpha_z\leqslant\dim Y$
for every point $z\in X$. This contradicts \eqref{rk}.
\quad $\square$
\renewcommand{\qed}{}\end{proof}

This prompts to explore smoothness of $\GG$. The
answer is known:

\begin{theorem}[{{\rm \cite[\hskip -.6mm\S3]{St4},\hskip .3mm\cite[\hskip -.6mm Prop.\,4.1]{R1}\!,\hskip .3mm\cite[\hskip -.6mm Prop.\!13.3]{R2}\!,\hskip .3mmRemark\,\ref{smo}\,below}}]\label{sing}
%%cf.
%%see also
%%and
%%Remark \ref{smo}}}]\label{sing}
%%below
%%}}]\label{sing}
%%Let ${\rm char}\,k\neq 2$.
The following properties are equivalent{\rm:}
\begin{enumerate}
\item[\rm(i)] $\GG$ is smooth{\rm;} \item[\rm(ii)] $\GG$ is
isomorphic to the  affine space ${\mathbf A}^r${\rm;}
\item[\rm (iii)] $G=G_1\times\cdots\times G_s$ where
every $G_i$ is either a simply connected simple
algebraic group or isomorphic to ${\bf SO}_{n_i}$ for an
odd $n_i$.
\end{enumerate}
\end{theorem}

%%\begin{remark} See also Remark  \ref{smo} below about the proof.
%%\end{remark}

\vskip 2mm

This criterion of smoothness of $\GG$ may be also
interpreted as that of the existence of $r$ generators
of the algebra of class functions on $G$. Below we
describe a minimal system of generators of this algebra
and singularities of $\GG$ in the general case. This
also yields a minimal system of generators of the
representation ring of $G$.

Let $B:=\tau(\widehat B)$ and $T:=\tau(\tT)$. They are
respectively  a Borel subgroup and a maxi\-mal torus of
$G$. We naturally identify ${\rm X}(\tT)$ with the lattice
${\rm X}(\tT)\otimes 1$ in ${\rm X}(\tT)\otimes_{\bf Z} \bf R$ and
view
%%the lattice
${\rm X}(T)$
%%of characters
%%of $T$
as a sublattice of ${\rm X}(\tT)$ identifying
%%${\rm X}(T)$
%%with its image under the embedding
%%$\big(\tau|_{\tT}\big)^{\hskip -.5mm*}$.
$\chi^\mu$  with $\big(\tau|_{\tT}\big)^{\hskip -.5mm*}\!(\chi^\mu)$
for $\mu\in {\rm X}(T)$.
%%\in {\rm X}(\tT)$.
%%As the index of ${\rm X}(T)$ in ${\rm X}(\tT)$ is finite,
%%we can (and shall) naturally identify  ${\rm X}(T)\otimes_{\bf Z} \bf %%R$ and ${\rm X}(\tT)\otimes_{\bf Z} \bf R$.
%%Then the index of ${\rm X}(T)$ in ${\rm X}(\tT)$
%%$[{\rm X}(\tT):{\rm X}(T)]$
%%is finite and
%%${\rm X}(\tT)$ is the weight lattice of
%%${\rm
%%X}(T)$.
%%the root system of
%%$G$ regarding  $T$.

Let
%%$W=N_{\tG}(\tT)/\tT$,
 %%where
 $N_{\tG}(\tT)$
 %%is
 be the normalizer of $\tT$ in $\tG$.
 %%be the
 %%The Weyl group of $\tG$.
 %%, i.e., the quotient
%%of $\tT$ in its normalizer, acting on $\tT$ via
%%It acts on $\tT$ via
%%conjugation. %%The Weyl group of $T$ is
%%naturally identified
%%naturally
%%with $W$, see
%%\cite[Prop.\,11.20 and Cor.\,11.11]{Bor}.
%%2(d) in 13.17
%%]{Bor}).
As $\tau$ is a central isogeny, the Weyl group $W$ of $T$ is
naturally identified
%%naturally
with  $N_{\tG}(\tT)/\tT$, see \cite[Prop.\;11.20]{Bor}.
%%$W$
As $W$ is finite, a categorical quotient
%%$(T/\!\!/W, \pi_{W, T})$
for the conjugating action of $W$ on $T$ is, in fact, geometric, so
%%and therefore
we denote
%%its
the corresponding quotient variety by $T/W$. Let\vskip -1mm
\begin{equation}\label{TW}
\pi_{W, T}^{\ }\colon T\to T/W
\end{equation}\vskip 2mm\noindent
be
%%its
the corresponding quotient morphism.

%%and
The root system $\Phi$ of $\widehat G$ regarding $\tT$
 (respectively, its positive part $\Phi_+$ regarding $\widehat B$) coincides with that of $G$ regarding $T$ (respectively, its positive part regarding $B$)
 %%, see
 \cite[Theorem 22.6(iii)]{Bor}. Let $$\alpha_1,\ldots,\alpha_r\in \Phi_+$$ be the set of all simple roots.
%%The
%%index of ${\rm X}(T)$ in ${\rm X}(\tT)$
%%$[{\rm X}(\tT):{\rm X}(T)]$
%%is finite and
%%lattice
%%${\rm X}(\tT)$ is
The weight lattice of $\Phi$ is ${\rm X}(\tT)$.
%%${\rm
%%X}(T)$.
%%the root system of
%%$G$ regarding  $T$.

The
monoid of highest weights of simple $\tG$-modules (regarding $\widehat T$
and $\widehat B$) is\vskip -1mm
\begin{equation}\label{wD}
\widehat{\mathcal D}:={\bf N}\varpi_1+\cdots+ {\bf N}\varpi_r,
\quad\mbox{where\hskip 2mm ${\bf N}=\{0, 1, 2,\ldots\}$.}
\end{equation}\vskip 2mm\noindent
 and
that of simple $G$-modules (regarding $T$ and $B$)
is\vskip -1mm
\begin{equation}\label{DDDD}
\mathcal D:=\widehat{\mathcal D}\cap {\rm X}(T).
\end{equation}\vskip 2mm

If $\varpi\in{\mathcal D}$ and $E(\varpi)$ is a simple
$G$-module with $\varpi$ as the highest weight, we
denote by ${\rm ch}_{\varpi}\in k[G]^{G}$ the character of
$E(\varpi)$.

Given
%%a nonzero commutative ring $A$ with identity
%%element and
a commutative monoid $M$, we denote by
%%$A[M]$
${\bf Z}[M]$  the semigroup ring of $M$ over
%%$A$
${\bf Z}$.
%%We
%%identify $A[M]$ with $A\otimes_{\bf Z}{\bf Z}[M]$ in the
%%natural way.
 If $S$ is a submonoid of the multiplicative monoid of
 %%$A[M]$
 ${\bf Z}[M]$ whose elements are linearly independent over
 %%$A$,
  ${\bf Z}$,
then the subring of
%%$A[M]$
${\bf Z}[M]$ generated by $S$ is
naturally identified
%%naturally
with
%%$A[S]$.
${\bf Z}[S]$. In particular, we consider
%%$%%A
${\bf Z}[{\rm X}(T)]$ and
%%$A
${\bf Z}[\mathcal D]$ as the subrings of
%%$A
${\bf Z}[{\rm X}(\tT)]$. The former is stable with respect to
the natural action of $W$ on
%%$A
${\bf Z}[{\rm X}(\tT)]$.
Following the notation
and terminology
of
%%{\sc Bourbaki}
\cite[VI.3.1]{Bou}, we
denote by $e^{\mu}$ the element of
%%$A
${\bf Z}[{\rm X}(\tT)]$ corresponding to $\mu\in {\rm
X}(\tT)$ and,  for any element\vskip -1mm
\begin{equation}\label{suppo}
x=\sum_{\mu\in {\rm X}(\tT)} a_{\mu}e^\mu\in
%%A
{\bf Z}[{\rm X}(\tT)] ,\qquad a_\mu\in {\bf Z},
%%A,
  \end{equation}\vskip 2mm\noindent
%%we
call
%%the set
$\{\mu\in {\rm X}(\tT)\mid a_\mu\neq 0\}$ the {\it support} of $x$. The nonzero summands $a_\mu e^\mu$ in  \eqref{suppo} are called
the {\it terms} of $x$.

Given an algebraic group $H$, we denote by $R(H)$ the
{\it representation ring} of $H$: its additive group is
the Grothendieck group of the category of finite
dimensional algebraic $H$-modules with respect to exact
sequences and the multiplication is induced by tensor
product of modules. Using $\tau$, we  identify $R(G)$ in
the natural way with the subring of~$R(\tG)$.

If $E$ is a finite dimensional algebraic $G$-module and
$E_\mu$ is its weight space of a weight $\mu\in {\rm
X}(T)$, then the formal character of $E$,\vskip -1mm
\begin{equation}\label{ch}
{\rm ch}_{G}[E]:=\sum_{\mu\in {\rm X}(T)}(\dim E_\mu)
e^{\mu},
\end{equation}\vskip 2mm\noindent
is an element of ${\bf Z}[{\rm X}(T)]^W$ depending only
on the class $[E]$ of $E$ in $R(G)$. Clearly,\vskip -1mm
\begin{equation}\label{tens}
{\rm ch}_{G}[E\otimes E']={\rm ch}_{G}[E]\,{\rm
ch}_{G}[E'].
\end{equation}\vskip 2mm

According to \cite[3.6]{Se2}, the homomorphism of ${\bf
Z}$-modules

\begin{equation}\label{chch}
{\rm ch}^{}_{G}\colon R(G)\longrightarrow {\bf Z}[{\rm
X}(T)]^W,\qquad [E]\mapsto {\rm ch}^{}_{G}[E],
\end{equation}
\vskip 2mm

\noindent is an isomorphism. By \eqref{tens} it is an
isomorphism of rings.

Next, we fix on ${\rm X}(\tT)\otimes_{\bf Z}{\bf R}$ the following partial order $\geqslant$:

%%defined by
%%determined
%%by $\widehat B$,
%%if $\mu$ and  $\nu$ are the elements of  ${\rm X}(\tT)\otimes_{\bf %%Z}{\bf R}$, then
\begin{equation}\label{geq}
\mu\geqslant \nu \hskip 1mm \iff \hskip 1mm \mu-\nu\in {\bf N}\alpha_1+\cdots+{\bf N}\alpha_r,\
%%quad \mbox{where \hskip 2mm $R_{\geqslant 0}=\{a\in {\bf R}\mid %%a\geqslant 0\}$}.
\end{equation}

\vskip 2mm
\noindent
cf.\;\cite[10.1]{Hum1}, \cite[31.2]{Hum}.  If $\mu$ is a maximal (with respect to $\geqslant$) element of the support of
%%an
element
\eqref{suppo},
%%$x\in
%%{\bf Z}[{\rm X}(\tT)]$,
then $a_\mu e^\mu$ is called the {\it maximal term} of $x$, cf.\;\cite[VI.3.2]{Bou}.

%%$\mu, \nu\in {\rm X}(\tT)\otimes_{\bf Z}{\bf R}$, the condition $\mu\geqslant \nu$ %%means that $\mu-\nu$ is
%%a linear combination of simple roots of $\tG$ regarding $\widehat T$ and $\widehat B$ %%with nonnegative coefficients.

\begin{definition}
Let $\varpi\in \widehat{\mathcal D}$. We say that an
element $x\in {\bf Z}[{\rm X}(\tT)]^W$ is $\varpi$-{\it
sharp} if
the following property (M) holds:
\vskip 3mm
\begin{enumerate}
\item[(M)]
$e^{\varpi}$ is
the unique maximal term of $x$.
%%The set of maximal terms of $x$ consists of a single element
%%and this element is $e^{\varpi}$.
\end{enumerate}
\end{definition}

\vskip -.5mm

\begin{example}\label{sharp} The elements
${\rm ch}^{}_{\tG}[E(\varpi)]$ and
\begin{equation}\label{S}
S(e^\varpi):=
\sum_{\mu\in W\cdot \varpi} e^{\mu}
%%\in  {\bf
%%Z}[{\rm X}(\tT)]^W.
\end{equation}
are $\varpi$-sharp
%%$S(e^\varpi)$ and ${\rm ch}^{}_{\tG}[E(\varpi)]$
%%are $\varpi$-sharp
(this follows, e.g., from
\cite[31.2, 31.3]{Hum}; cf. also \cite[VI.3.4]{Bou}).
%%, see also the proof of
%% \cite[VI.1.6, Prop.\;18]{Bou}).
\quad $\square$
\end{example}

By \eqref{geq} property (M) implies that the support of a
$\varpi$-sharp element $x$ lies in $\varpi+{\rm X}(T)$.
This and \cite[VI.3.4, formula (6)]{Bou} yield the equality
\begin{equation}\label{SSS}
x=S(e^\varpi)+\mbox{sum of some of the elements $\pm S(e^{\varpi'})$ with
$\varpi'\in \widehat{\mathcal D}$, $\varpi'<\varpi$}.
\end{equation}

\noindent By \cite[VI.3.2, Lemma 2]{Bou} if $x$ is $\varpi$-sharp and
$x'$ is $\varpi'$-sharp, then $xx'$ is
$(\varpi+\varpi')$-sharp.

Now, fix a $\varpi_i$-sharp element $x_{\varpi_i}\in {\bf
Z}[{\rm X}(\tT)]^W$, $i=1,\ldots, r$, and  put

\begin{equation*}
x_\varpi:= x_{\varpi_1}^{m_1}\cdots
x_{\varpi_r}^{m_r}\qquad\mbox{for}\qquad
\varpi=m_1\varpi_1+\cdots+m_r\varpi_r\in \widehat{
\mathcal D}.
\end{equation*}
\vskip 2mm

By \cite[VI.3.4, Theorem 1]{Bou} the set
$\{x_{\varpi}\mid \varpi\in \widehat{ \mathcal D}\}$ is
then a basis of the ${\bf Z}$-module ${\bf Z}[{\rm
X}(\tT)]^W\!.$ As $\{e^\mu\mid \mu\in {\rm X}(T)\}$ is a
basis of the ${\bf Z}$-module ${\bf Z}[{\rm X}(T)]$ and
the support of $x_\varpi$ lies in $\varpi+{\rm X}(T)$,
we deduce from this and \eqref{DDDD} that
$\{x_{\varpi}\mid \varpi\in\mathcal D\}$ is a basis of
the ${\bf Z}$-module ${\bf Z}[{\rm X}(T)]^W$. Hence the
homomorphism of the ${\bf Z}$-modules

\begin{equation}\label{psi}
\vartheta\colon {\bf Z}[{\rm X}(T)]^W\longrightarrow {\bf
Z}[\mathcal D], \qquad \vartheta(x_\varpi)=e^\varpi\quad
\mbox{for\quad $ \varpi\in \mathcal D$},
\end{equation}
\vskip 2mm

\noindent is an isomorphism. Since
$x_{\varpi+\varpi'}=x_\varpi x_{\varpi'}$, it is, in
fact, an isomorphism of rings.

As
%%by {\sc Dedekind}'s theorem
$\{%%f_\mu\colon T\to k,
%%t\mapsto t^\mu
\chi^\mu\mid \mu\in {\rm X}(T)\}$ is a $k$-basis of
%%the vector space
$k[T]$
%%over $k$,
%%(see, e.g.,
(cf.\;\cite[3.2.3]{Sp}) and $\chi^\mu \chi^\nu=\chi^{\mu+\nu}$,
the
$k$-linear
map

\begin{equation}\label{TX}
k[T]\longrightarrow k\otimes_{\bf Z}{\bf Z}[{\rm X}(T)],\qquad \chi^\mu\mapsto 1\otimes e^\mu,
\end{equation}

\vskip 2mm
\noindent
is an isomorphism of $k$-algebras. As this isomorphism is $W$-equivariant, its restriction to
$k[T]^W$ is an isomorphism of $k$-algebras

\begin{equation}\label{eta}
\eta\colon k[T]^W\longrightarrow \big(k\otimes_{\bf Z}{\bf Z}[{\rm X}(T)]\big){}\!^W=k\otimes_{\bf Z}{\bf Z}[{\rm X}(T)]^W
\end{equation}
\vskip 2mm
\noindent
(regarding the latter equality, see, e.g., \cite[VI.3.4]{Bou} or \cite[Prop. 3.3.1]{L}).

%%We identify them by
%%means of this isomorphism. So we have $k[T]=k[{\rm
%%X}(T)]$ and \ \vskip -4mm
%%\begin{equation}\label{T/W}
%%k[T/W]=k[T]^W=
%%k[{\rm X}(T)]^W.
%%\end{equation}

Finally, take into account that by \cite[6.4]{St1} the
%%composition of the
%%following
%%restriction
%%maps
map
\begin{equation}\label{G//G}
{\rm res}\colon
%%k[\GG]
k[G]^G\longrightarrow
%%=
%%\xrightarrow{\pi_G^*}k[G]^G\xrightarrow{\rm res}
k[T]^W,\qquad f\mapsto f|_T,
%%\xrightarrow{(\pi_{W, T}^*)^{-1}}k[T/W],\quad \mbox{where ${\rm res}(f):=f|_T$},
\end{equation}
\vskip 2mm

\noindent is an isomorphism of $k$-algebras.
%%We denote it by
%%\begin{equation*}
%%\zeta:=(\pi_{W, T}^*)^{-1}\circ{\rm res}\circ\pi_G^*
%%\end{equation*}

Summing up,
we obtain the following

\begin{theorem} \label{toric}
\
\begin{enumerate}
%%\item[\rm(i)] $\GG$ and $T/W$ are affine toric
%%varieties of $\;T$ whose algebras of regular functions
%%are iso\-morphic to $k[\mathcal D]$.
\item[\rm(i)] All the maps in the diagram
%%\
%%
\vskip -2mm
%%
%%\
%%the
%%diagram
%%\ \vskip -3mm
$$
k[\GG]\xrightarrow{\pi_{G}^*}k[G]^G\xrightarrow{{\rm res}}
%%\zeta
%%}
%%k[T/W]
k[T]^W
\xrightarrow{\eta}k\otimes_{\bf Z}{\bf Z}[{\rm X}(T)]^W
\xrightarrow{{\rm
id}\otimes \vartheta} k\otimes_{\bf Z}{\bf Z}[\mathcal D]
$$
\vskip 2mm
%%\noindent
{\rm(}see {\rm \eqref{piG}, \eqref{G//G}, \eqref{eta},
\eqref{psi}}{\rm)}
%%all maps
are isomorphisms of
$k$-algebras.
\item[\rm(ii)]  Let $F$ be
%%the simple
a subring of $k$. Then the image of $F\otimes_{\bf Z}\!R(G)$
in
%%$k[G]^G$
$k[\GG]$ un\-der the composition of the following ring
isomorphisms
%%\
%%
%%\vskip -2mm
%%
%%\
\end{enumerate}\vskip 1mm
\begin{equation}
\label{composition}
k\otimes_{\bf
Z}R(G)\xrightarrow{{\rm id}\otimes {\rm ch}^{}_G} k\otimes_{\bf Z}{\bf Z}[{\rm
X}(T)]^W\xrightarrow{\eta^{-1}}k[T]^W \xrightarrow{{\rm res}^{-1}}k[G]^G\xrightarrow{(\pi_{G}^*)^{-1}} k[\GG]
\end{equation}
\vskip 4mm
%%\noindent
\begin{enumerate}
\item[]{\rm(}see {\rm \eqref{chch}, \eqref{eta}, \eqref{G//G}, \eqref{piG}}{\rm)} is an $F$-form of $\;k[\GG]$.
%% isomorphic to $F\otimes_{\bf
%%Z}R(G)$.
In particular, if
%%$F$ is a simple subring of $k$ and
$\;{\rm char}\,k=0$,
then $R(G)$
%%this image
is a ${\bf Z}$-form of $k[\GG]$.
%% that is isomorphic to $R(G)$.
%%that is iso\-morphic to $R(G)$.
\end{enumerate}
\end{theorem}

%%\begin{remark} {\rm The fact that ``multiplicative invariants'' of finite reflection groups are %%semigroup algebras is already in the literature, first implicitly, then explicitly, see
%%the historical account in
%%\cite[Int\-roduction]{L0}. Essentially, the main
%%ingredients date back to  \cite[\hskip -.5mm\S6]{St1}
%%and \cite[\hskip -1mm VI, \S3]{Bou}.}
%%\end{remark}

 Recall (see, e.g., \cite{F}) that there is a bijection between
 $T$-isomorphism classes of
 affine toric varieties of $T$
 %%(considered up to
 %%$T$-equivariant isomorphism)
  and $r$-dimensional
  rational convex polyhedral
 cones in ${\rm X}(T)\otimes_{\bf Z} {\bf R}$ (i.e.,
 convex cones
 %%in ${\rm X}(T)\otimes_{\bf Z} {\bf R}$
 with apex at the origin generated by a finite number of elements of ${\rm X}(T)$):
 %%):
 %%it %%assigns
it juxtaposes to a cone $\mathcal C$ the affine variety
 $Y_{\mathcal C}$
 %%whose coordinate algebra is
 such that $k[Y_{\mathcal C}]:=k\otimes_{\bf Z}{\bf Z}[{\mathcal C}\cap {\rm X}(T)]$ endowed with the natural action of $T$. As $w\cdot {\rm X}(T)={\rm X}(T)$ for any element $w\in W$,  the varieties $Y_{\mathcal C}$ and  $Y_{w\cdot \mathcal C}$ are isomorphic  (but not $T$-isomorphic if $w\neq e$). As $W$ acts transitively on the set of Weyl chambers of $\Phi$, Theorem \ref{toric}(i) and formula \eqref{DDDD} then yield the following
 %%; the action of $T$ on $X_{\mathcal C}$ is determined by the natural %% $T$-module structure of .

\begin{corollary}  \label{tv} $\GG$ and $T/W$ are isomorphic to the {\rm(}underlying variety of \hskip 0mm{\rm)}\;af\-fine toric variety $Y_{\mathcal C}$, where
$\mathcal C$ is any Weyl chamber of the root system
$\,\Phi$.
%%determining the basis $\alpha_1,\ldots, \alpha_r$ of $\;\Phi$.
%%Then
%%$\GG$ and $T/W$ are isomorphic to
%%the affine toric
%%variety $X_{\mathcal C}$ of $\;T$ is isomorphic
%%where $\sigma$ is the Weyl chamber
%%$\;\mathcal D$.
%%${\bf R}_{\geqslant 0}\varpi_1+\cdots+{\bf R}_{\geqslant %%0}\varpi_r$.
\end{corollary}

\begin{corollary} \label{toricstr} There are actions of $\;T$ on $\GG$ and $T/W$ with
%%the
dense open orbits.
%% and unique fixed points.
\end{corollary}

\begin{remark} {\rm This action of $T$ on $T/W$ cannot be lifted to $T$
making
%%the
quotient mor\-phism
%%$\pi^{}_{W, T}\colon T\to T/W$
 \eqref{TW} equivariant, see below Subsection 6.C.}
\end{remark}

Since toric varieties are rational,
Corollary \ref{tv} yields

\begin{corollary}  \label{ratio}
$\GG$ and $T/W$ are rational varieties.
\end{corollary}

\begin{remark} {\rm The fact that ``multiplicative invariants'' of finite reflection groups are semigroup algebras is already in the literature, first implicitly, then explicitly, see
the historical account in
\cite[Int\-roduction]{L0}. Essentially, the main
ingredients date back to  \cite[\hskip -.5mm\S6]{St1}
and \cite[\hskip -1mm VI, \S3]{Bou}.}
\end{remark}

In the next statement Theorem \ref{toric} is applied to
finding a minimal system of ge\-ne\-rators of the
algebra $k[G]^G$ and that of the ring $R(G)$.

Let $\mathcal H$ be the Hilbert basis of the monoid
$\mathcal D$, i.e., the set of all its indecomposable
elements:\vskip -1mm
%%\ \vskip -5mm
\begin{equation}\label{Hilbert}
\mathcal H=\mathcal D_+\setminus 2 \mathcal D_+
\quad\mbox{where}\quad \mathcal D_+:=\mathcal D\setminus
\{0\}, \quad 2\mathcal D_+:=\mathcal D_+ + \mathcal D_+.
\end{equation}\vskip 2mm\noindent
The set $\mathcal H$ is finite, generates $\mathcal D$,
and every generating set of $\mathcal D$ contains
$\mathcal H$ (see, e.g., \cite[3.4]{L}).

\begin{remark}
{\rm There is an  algorithm for efficient computing $\mathcal H$,
see \cite[13.2]{Stu} (see also Ex\-amp\-le \ref{singu}
below)}.
\end{remark}

\begin{theorem}\label{generators} \

\begin{enumerate}
\item[\rm(i)] The cardinality of every generating set of
the algebra $k[G]^G$ of class functions on $G$ is not
less than the cardinality of $\;\mathcal H$. The same
holds for every generating set of the representation
ring $R(G)$ of $G$. \item[\rm(ii)] $\{[E(\varpi)] \mid
\varpi\in \mathcal H\}$ is a
generating set of the ring $R(G)$. \item[\rm(iii)]
$\{{\rm ch}_{\varpi}\mid \varpi\in \mathcal H\}$ is a
generating set of the algebra
$k[G]^G$.
\end{enumerate}
\end{theorem}
\begin{proof}
(i) Let $Y$ be the affine toric variety of $T$ with
$k[Y]=k\otimes_{\bf Z}{\bf Z}[\mathcal D]$. The linear span $I$ of
$\{1\otimes e^\varpi\mid \varpi\in\mathcal D_+\}$ over $k$ is a
maximal $T$-invariant ideal in $k[Y]$. Hence $I/I^2$ is
the cotangent space of $Y$ at the $T$-fixed point $y$
where $I$ vanishes. As $I^2$ is the linear span of
$\{1\otimes e^\varpi\mid \varpi\in 2\mathcal D_+\}$ over $k$,
this and \eqref{Hilbert} yield the equalities
\begin{equation}\label{tange}
\dim {\rm T}_{y, Y}=\dim I/I^2= |\mathcal H|.
\end{equation}

Now take into account that, given an affine algebraic
variety $X$, the algebra $k[X]$ can be generated by $d$
elements if and only if $X$ admits a closed embedding in
${\bf A}\!^d$. Hence $d\geqslant \dim {\rm T}_{x, X}$
for every point $x\in X$. This, Theorem
\ref{toric},
%%(i),(iii),
and \eqref{tange} prove~(i).

\vskip 1mm

(ii) Let $\mu\in \mathcal D$. As $\mathcal H$ generates
$\mathcal D$, there is a decomposition
$$
\mu=\sum_{\varpi\in\mathcal H}a_{\varpi}\varpi,\qquad \mbox{where}\quad a_\varpi\in{\bf N}.
$$
Hence, by Example \ref{sharp},
\begin{equation}\label{Mmu}
M^\mu:=\prod_{\varpi\in\mathcal H}\big({\rm ch}^{}_G[E(\varpi)]\big)^{a_{\varpi}}
\end{equation}
%% that
%%there is
is a $\mu$-sharp element of ${\bf Z}[{\rm X}(T)]^W$.
%%monomial $M^\mu$  in the elements
%%of the set $\{{\rm ch}^{}_G[E(\varpi)]\mid \varpi\in
%%\mathcal D\}$.
By \eqref{SSS} we have\vskip -1mm
\begin{equation}\label{MM} M^\mu=S(e^\mu)+
\mbox{sum of some of the elements $\pm S(e^{\mu'})$ with $\mu'\in
{\mathcal D}$, $\mu'<\mu$}.
\end{equation}
\vskip 2mm
\noindent But $\{S(e^\mu)\mid \mu\in \mathcal D\}$ is a
basis of the ${\bf Z}$-module ${\bf Z}[{\rm X}(T)]^W$
(see \cite[VI.3.4, Lemma 3]{Bou}).  By \cite[VI.3.4,
Lemma 4]{Bou}   we then deduce from \eqref{MM} that the
set $\{M^\mu\mid \mu\in\mathcal D\}$ generates the ${\bf
Z}$-module ${\bf Z}[{\rm X}(T)]^W$. This and \eqref{Mmu} imply that the
ring ${\bf Z}[{\rm X}(T)]^W$ is generated by the set
$\{{\rm ch}^{}_{G}[E(\varpi)]\mid \varpi\in\mathcal
H\}$. As \eqref{chch} is an isomorphism of rings, this
proves (ii).

\vskip 1mm

(iii) It follows from (ii) that the set $\{1\otimes
[E(\varpi)]\mid \varpi\in\mathcal H\}$ generates the
ring $k\otimes_{\bf Z}R(G)$. But formula \eqref{ch}
shows that ${\rm ch}_\varpi$ is the image of $1\otimes
E[(\varpi)]$ under the isomorphism
${\rm res}^{-1}\circ \eta^{-1}\circ ({\rm id}\otimes {\rm ch}^{}_G)$
%%composition of the ring
%%isomorphisms in
(see diagram \eqref{composition}). This proves
(iii). \quad $\square$
\renewcommand{\qed}{}\end{proof}

%%Theorem \ref{toric}(i),
%%Corollary \ref{toricstr},
 The proof of Theorem \ref{generators} and formula \eqref{tange} yield
 the following
\begin{corollary} \label{TfGG} %%The dimension of tangent spaces
%%of
%%the variety
%%$\GG$
%%If $x$ runs through the points of $\GG$, then
The maximum of
the
function $x\mapsto \dim\,{\rm T}_{x, \GG}$ is equal to $|{\mathcal H}|$.
%%and
%%attained
%%its maximum equal to $|{\mathcal H}|$
%%at a single  point $\;v_0$
%%$x_0$
%%of $\GG$.
%%This maximum
%%at the unique
%%fixed point of $T$
%%is equal to $|{\mathcal H}|$ and $x_0$ is a fixed of $T$
%%with respect to the action specified in Corollary {\rm\ref{toricstr}}.
\end{corollary}

%%According to
 In line with the general theory of toric varieties, as the Weyl chambers are simpli\-cial cones,
%%Theorem
%%\ref{toric}(i)
Corollary \ref{tv} implies, at least for ${\rm char} \,k=0$,
that $\GG$ and $T/W$ are isomorphic to the quotient of
${\bf A}^r$ by a linear action of a certain finite
abelian group and hence,  in
particular, $\GG$ and $T/W$ may have only finite
quotient singularities
%%(see, e.g.,
\cite[Prop.\;1.25]{O}.
%%).
Below, for arbitrary ${\rm char} \,k$ and separable $\tau$, we prove the existence of
%%
%%Below
such a finite group and its
action on ${\bf A}^r$
%%giving their
by means of their explicit description. This yields an explicit description of
singularities of $\GG$ and that of the minimal generating sets
 of the algebra of class functions on $G$ and of
 the representation ring
of $G$.
%%by
%%providing
%%means of their
%%explicit describing.
%%are explicitly described.
%%We impose
%%assuming
%%that ${\rm char} \,k$ is
%%no constraints on ${\rm char} \,k$, but assume that $\tau$ is %%separable.

 The
 %%latter
 assumption that $\tau$ is separable means  that there is a subgroup
$Z$ of $\tC$ such that $G=\tG/Z$ and $\tau$ is the
quotient morphism $\tG\to \tG/Z$. In this situation we
have \vskip 1mm
\begin{equation}\label{varpiX} {\rm X}(T)=\{\mu\in
{\rm X}(\tT)\mid c^\mu=1\ \mbox{for every $c\in
Z$}\}.\end{equation}
\vskip 2mm

For the action of $\tT$ on ${\bf A}^r$ defined by formula \eqref{action}, consider the $\tT$-orbit map of the point
$(1,\ldots,1)\in{\bf A}^r$:
\begin{equation}\label{iot}
\iota\colon \tT\longrightarrow {\bf A}^r,\qquad
\iota(t)=t\cdot (1,\ldots, 1),
\end{equation}
and identify $k[\tT]$ with $k\otimes_{\bf Z}{\bf Z}[{\rm X}(\tT)]$ by means of the isomorphism
\begin{equation*}
k[\tT]\longrightarrow k\otimes_{\bf Z}{\bf Z}[{\rm X}(\tT)],\qquad \chi^\mu\mapsto 1\otimes e^\mu.
\end{equation*}

The map $\iota^*\colon k[{\bf A}^r]\to
%%k[\tT]$
%%k[{\rm
%%X}(\tT)]
k\otimes_{\bf Z}{\bf Z}[{\rm X}(\tT)]$
is an embedding as $\iota$ is dominant by
Lemma \ref{lin}(i). Let $$y_1,\ldots, y_r$$ be the
standard coordinate functions on ${\bf A}^r$. Then
\eqref{action} and \eqref{iot} yield the equality\vskip 1mm
\begin{equation}\label{embed}
\iota^*(y_i)=1\otimes e^{\varpi_i}.\end{equation}\vskip 2mm
%%\noindent

From \eqref{action} we deduce that $k[{\bf
A}^r]^Z$ is the linear span over $k$ of all the monomials
$y^{m_1}\cdots y^{m_r}$ with $m_1,\ldots, m_r\in {\bf
N}$ such that $c^{m_1\varpi_1+\cdots+m_r\varpi_r}=1$ for
every $c\in Z$. By \eqref{varpiX} the latter condition
is equivalent to the inclusion
$m_1\varpi_1+\cdots+m_r\varpi_r\in {\rm X}(T)$. This,
\eqref{embed}, \eqref{wD}, and \eqref{DDDD} imply the equality\vskip 1mm
$$\iota^*\big(k[{\bf A}^r]^Z\big)=k\otimes_{\bf Z}{\bf Z}[\mathcal D].$$\vskip 2mm

Since $Z$ is finite, a categorical quotient for the action of $Z$ on ${\bf A}^r$ is geometric and so we denote  the corresponding quotient variety by ${\bf A}^r/Z$. Thus, taking
into account Theorem \ref{toric}, we obtain the
following isomorphisms of $k$-algebras:\vskip -1mm
\begin{equation*}\label{isoms}
\xymatrix{&&&
k[T/W]\ar[d]^{\pi_{W, T}^*}&&
\\
k[{\bf A}^r]^Z\ar[r]^{{\iota}^*}&
k\otimes_{\bf Z}{\bf Z}[\mathcal D]
%%\ar[r]^{\eta^{-1}\circ({\rm id}\otimes \vartheta)^{-1}}
&&
k[T]^W
\ar[ll]_{\hskip 4mm({\rm id}\otimes \vartheta)\circ\eta}
%%\ar[r]^{({\rm res})^{-1}}
&
 k[G]^G
 \ar[l]_<(0.24){{\rm res}}
 %%\ar[r]^{(\pi_G^*)^{-1}}
 &
 k[\GG]\ar[l]_{\pi_G^*}.
}
%%k[{\bf A}^r]^Z\xrightarrow{{\iota}^*} k\otimes_{\bf Z}{\bf %%Z}[\mathcal D]\xrightarrow{\eta^{-1}\circ({\rm id}\otimes %%\vartheta)^{-1}} k[T]^W
%%\xrightarrow{({\rm res})^{-1}} k[G]^G.
\end{equation*}\vskip 2mm\noindent
They, in turn, induce the following isomorphisms of varieties\vskip -1mm
$$\GG\xrightarrow{\simeq} T/W\xrightarrow{\simeq}{\bf A}^r/Z.$$\vskip 2mm

By means of a special parametrization of $\tT$ one
can obtain
an explicit description of the elements of $\tC$
well adapted for computing $k[{\bf A}^r]^Z$. Since
$\tG=\tG_1\times\cdots\times \tG_s$ and $\widehat
C=\widehat C_1\times\cdots\times \widehat C_s$ where
every $\widehat G_i$ is a nontrivial  normal simply
connected simple subgroup of $\tG$ and $\widehat C_i$ is
the center of $\widehat G_i$, it suffices to describe
this parametrization for simple groups $\tG$. The answer
is given below in Lemma \ref{explcenter}.

Namely,
%%let
%%${\alpha}_1,\ldots,{\alpha}_r\in
%%{\rm X}(\widehat T)$ be the system of simple roots of
%%$\widehat T$ regarding  $\widehat  B$ and
let
${\alpha}^\vee_i\colon {\bf G}_m\to \widehat T$
be the coroot corresponding to ${\alpha}_i$.
Then, for every $s\in {\bf G}_m$, we have
\begin{equation}\label{pairr}
\big({\alpha}^\vee_i(s)\big){}^{\varpi_j}=
\begin{cases} s &\mbox{if $i=j$},\\
1 &\mbox{if $i\neq j$.}
\end{cases}
\end{equation}

If $\langle\ {,}\ \rangle$ is the natural pairing
between the lattices of characters and cocharacters of
$\widehat T$, we put
$$n_{ij}:=\langle{\alpha}_i,{\alpha}_j^\vee
\rangle.$$
So
$(n_{ij})_{i,j=1}^r$ is the Cartan matrix
of $\tG$.

 By \cite[Lemma 28(b),(d) and its Cor.\,(a)]{St2} the map\vskip -1mm
\begin{equation}\label{toruss}
\nu \colon {\bf G}_m^{r} \longrightarrow \widehat T,
\qquad \nu(s_1,\ldots, s_r)=
{\alpha}^\vee_1(s_1)\cdots
{\alpha}^\vee_r(s_r),
\end{equation}\vskip 2mm
\noindent
is an isomorphism of groups and\vskip -1mm
\begin{equation}\label{center}
\widehat C=\{{\alpha}^\vee_1(s_1)\cdots
{\alpha}^\vee_r(s_r)\mid s_1^{n_{i1}}\cdots
s_r^{n_{ir}}=1\hskip 2mm\mbox{for every $i=1,\ldots,
r$}\}.
\end{equation}\vskip 2mm

By \eqref{pairr} and \eqref{toruss}  we have\vskip -1mm
\begin{equation*}
\big(\nu(s_1,\ldots,s_r)\big)^{\varpi_i}=
\big({\alpha}^\vee_1 (s_1)\big){}^{\varpi_i}
\cdots
\big({\alpha}^\vee_r(s_r)\big){}^{\varpi_i}=s_i.
\end{equation*}\vskip 2mm\noindent
This and \eqref{action} imply that, for every
$s=(s_1,\ldots,s_r)\in {\bf G}_m^r$ and $(a_1, \ldots,
a_r)\in{\bf A}^r$, the following equality holds:
\begin{equation}\label{actiontorus}
\nu(s)\cdot(a_1,\ldots, a_r)=(s_1a_1,\ldots, s_ra_r).
\end{equation}

\begin{lemma}\label{explcenter}
For every simple simply connected group $\tG$, the
subgroup $\nu^{-1}(\widehat C)$ of
the torus $\;{\bf G}_m^r$ is
described in the following Table $1$ {\rm(}simple roots
in \eqref{toruss} are
%%numerated
numbered as in {\rm\cite{Bou}}{\rm):}

\vskip 3mm

\centerline{\sc Table $1$.}

\vskip 3mm
\begin{center}
\begin{tabular}{c|cc}
\text{\rm type of} $\tG$ && $\nu^{-1}(\widehat C)$
\\[2pt]
\hline \hline
&\\[-9pt]
${\sf A}_r$&& $\{(t, t^2, t^3,\ldots, t^r)\mid
t^{r+1}=1\}$
\\[3pt]
%%%%%%%%%%%%%%%%%%%%%%%%%%%
 \hline
&\\[-8pt]
 ${\sf B}_r$ %%$l\geqslant 3$
 &&
 $\{(1,\ldots, 1, t)\mid t^2=1\}$
\\[3pt]
%%%%%%%%%%%%%%%%%%%%%%%%%%
\hline
&\\[-8pt]
${\sf C}_r$
%%$l\geqslant 2$
&& $\{(t, 1, t, 1,\ldots, t^{r\,{\rm mod}\,2})\mid
t^2=1\}$
\\[3pt]
%%%%%%%%%%%%%%%%%%%%%%%%%%%%
\hline
&\\[-8pt]
${\sf D}_r,$ $\mbox{$r${\fontsize{9pt}{5mm}
\selectfont\rm odd}}$ && $\{(t^2, 1, t^2, 1,\ldots, t^2,
t, t^{-1})\mid t^4=1\}$
\\[3pt]
%%%%%%%%%%%%%%%%%%%%%%%%%%%%
\hline
&\\[-8pt]
${\sf D}_r,$ $\mbox{$r${\fontsize{9pt}{5mm}
\selectfont\rm even}}$ && $\{(t_1, 1, t_1, 1,\ldots,t_1, 1,
t_1t_2, t_2)\mid t_1^2=t_2^2=1\}$
%%$\{(1, 1, 1, 1,\ldots, 1, s, s)\mid s^2=1\}$
\\[3pt]
%%%%%%%%%%%%%%%%%%%%%%%%%%%%%%
\hline
&\\[-8pt]
${\sf E}_6$&& $\{(t, 1, t^{-1}, 1, t, t^{-1})\mid
t^3=1\}$
\\[3pt]
%%%%%%%%%%%%%%%%%%%%%%%%%%%%%%%%%%
\hline
&\\[-8pt]
${\sf E}_7$&& $\{(1, t, 1, 1, t, 1, t)\mid t^2=1\}$
\\[3pt]
%%%%%%%%%%%%%%%%%%%%%%%%%%
\hline
&\\[-8pt]
${\sf E}_8$&& $\{(1,1,1,1,1,1,1,1)\}$
%%$\{e\}$
\\[3pt]
%%%%%%%%%%%%%%%%%%%%%%%%
\hline
&\\[-8pt]
${\sf F}_4$&& $\{(1,1,1,1)\}$
%%$\{e\}$
\\[3pt]
%%%%%%%%%%%%%%%%%%%%%%%%%
\hline
&\\[-8pt]
${\sf G}_2$&& $\{(1,1)\}$
\end{tabular}
\end{center}
 \end{lemma}

\vskip 1mm

\begin{proof} By \eqref{center} an element $(s_1,\ldots, s_r)\!\in\! {\bf G}_m^r$ lies in $\nu^{-1}(\widehat C)$ if and only if $(s_1,\ldots,s_r)$
%%$x_1=t_1, \ldots, x_r=t_r$
is a solution of the following system of equations:
\begin{equation}\label{systemm}
\left.
\begin{split}
&x_1^{n_{11}}\cdots x_r^{n_{1r}}=1,\\[-2pt]
&\hskip 1mm .\ . \ . \ .\ .\ .\ .\ .\ .\ .\ . \\[1pt]
&x_1^{n_{r1}}\cdots x_r^{n_{rr}}=1.
\end{split}\right\}
\end{equation}
%%where $(n_{ij})_{i,j=1}^r$ is the Cartan matrix of
%%$\tG$.

Let, for instance, $\tG$ be of type ${\sf D}_r$ for even
$r$. Using the explicit form of the Cartan matrix
\cite[Planche IV]{Bou}, one immediately verifies that
every element of $C':=\{(t_1, 1, t_1, 1,\ldots,t_1, 1, t_1t_2,
t)\mid t_1^2=t_2^2=1\}$ is a solution of \eqref{systemm}.
Hence, $C'\subseteq \nu^{-1}(\widehat C)$. On the other
hand, the fundamental group  of the root system of type
${\sf D}_r$ is isomorphic to ${\bf Z}/2{\bf Z}\oplus
{\bf Z}/2{\bf Z}$ (see \cite[8.1.11]{Sp} and
\cite[Planche IV]{Bou}). Hence, the {\sc Smith} normal
form of $(n_{ij})_{i,j=1}^r$ is ${\rm
diag}\,(1,\ldots,1,2,2)$. Therefore, there is a basis
$\beta_1,\ldots, \beta_r$ of the coroot lattice of $\tT$
such that, for $(s_1,\ldots,s_r)\in {\bf G}_m^r$, we
have $\beta_1(s_1)\cdots\beta_r(s_r)\in \widehat C$ if
and only if $(s_1,\ldots,s_r)$ is a solution of the
following system of equations:
\begin{equation*}
x_1=1,\ldots, x_{r-2}=1, x_{r-1}^2=1, x_{r}^2=1.
\end{equation*}
This yields the equality $|C'|=|\widehat C|$; whence $C'=\nu^{-1}
(\widehat C)$.

For the groups of the other types the proofs are
similar. \quad $\square$
\renewcommand{\qed}{}\end{proof}

The following examples illustrate how
this can be applied to  exploring  singularities of
$\GG$ and finding the minimal generating sets
$\{{\rm ch}_{\varpi}\mid \varpi\in \mathcal H\}$ and
$\{[E(\varpi)] \mid \varpi\in \mathcal H\}$ of,
respectively ,  the algebra of class functions on $G$ and
 the representation ring
of $G$.

\begin{examples}\label{singu}
\

(1) If  $Z$ is trivial, i.e., $G=\tG$, then we have
$$\mathcal H=\{\varpi_1,\ldots, \varpi_r\}$$
\noindent and $\GG$ is isomorphic to $\mathbf A^r$.

\vskip 1.5mm

(2) Let $\tG$ be of type ${\sf A}_r$.  If ${\rm char}\,k>0$, let
$({\rm char}\,k)^d$ be the maximal power of
${\rm char}\,k$
dividing $r+1$.
Put
%%$p={\rm char}\,k$ and
%%\vskip 1mm
$$
m:=
\begin{cases} r+1& \mbox{if %%$p=0$
${\rm char}\,k=0$,}\\
%%(r+1)/%%p
(r+1)/({\rm char}\,k)^d & \mbox{if
%%$%%p
${\rm char}\,k>0$.}
%% and
%%$
%%p
%%$({\rm char}\,k)^d$ is the maximal power of
%%$%%p
%%${\rm char}\,k$ %%that divides
%%dividing $r+1$}.
\end{cases}
$$\vskip 2mm\noindent
Then there are precisely $m$
%%the number of
different $m$th roots of unity in $k$
%%is
%%$m$
and from Table 1 we deduce that ${\widehat C}$ is a cyclic group of order $m$. Assume that ${\widehat C}$ is nontrivial, i.e., $m\geqslant 2$, and consider the case $Z={\widehat C}$, i.e.,
\begin{equation*}\label{pgl}
G=\tG/{\widehat
C}={\bf PGL}_{r+1}.
\end{equation*}

Take an element $z\in{\widehat C}$, $g\neq e$. As
$z=\nu((t, t^2, t^3,\ldots, t^r))$ where $t\neq 1$,
$t^{r+1}=1$,  formula
\eqref{actiontorus} implies that $z$ acts on ${\bf A}^r$ as a pseudo-reflection (i.e., $\dim_k({\bf A}^r)^{z}=r-1$) if and only if $r=1$. As is known \cite{Se3} (cf.\;also \cite[Theorem 7.2.1]{Ben}) if a quotient variety of ${\bf A}^r$ by a finite linear group is smooth, then this group is generated by pseudo-reflections
(and this quotient variety is isomorphic to ${\bf A}^r$).
Hence,  if  $r\geqslant 2$, then
%%for $G$
%%=\tG/{\widehat
%%C}={\bf PGL}_{r+1}$,
%%the variety
$\GG$ has singular points (this agrees with Theorem  \ref{sing}); if $r=1$, then $\GG$ is smooth, see the next example.

Actually, %%one
%%we get
our analysis provides a more precise information. Namely,
let $\mmu_m$
%%$ be
be the cyclic group ${\bf Z}/m{\bf Z}$. Fix a choice of generator
$g$ of $\mmu_m$ and primitive $m$th root of unity $\zeta$ in $k$.
Let $L$ be a one-dimensional $\mmu_m$-module on which $\mmu_m$ acts by means of the character $g^h\mapsto \zeta^h$. Put \begin{equation*}
V=\bigoplus_{i=1}^{r}L^{\otimes i}
\end{equation*}
(thus, if ${\rm char}\,k\nmid (r+1)$, then
%%does not divide $r+1$,
the representation of
$\mmu_{r+1}$ in $V$
is the reduced regular representation of  $\mmu_{r+1}$,
i.e., the quotient of
%%the
regular representation
%%of
 %%$\mmu_{r+1}$
 by the
 %%its
 unique one-dimensional trivial subrepresentation). %%
%%
%%
%%\vskip 20mm
%%
%%
%%Assume, first, that $p$ does not di\-vi\-de $r+1$.
%%%%Consider the cyclic group $%%{\mbox{\boldmath$\mu$}}
%%%%\mmu_{r+1}:={\bf Z}/(r+1){\bf Z}$ and let
%%Let $V$ be the sum of all %%its
%%(one-dimensional) nontrivial simple $\mmu_{r+1}$-modules, each %%taken with
%%multiplicity 1 (thus, the representation of
%%%%${\bf Z}/(r+1){\bf Z}$
%%$\mmu_{r+1}$ in $V$
%%is the quotient of the regular representation of
%%%%${\bf Z}/(r+1){\bf Z}$
%% $\mmu_{r+1}$ by its unique one-dimensional trivial %%subrepresentation).
Then
%%$G=\tG/{\widehat
%%C}={\bf PGL}_{r+1}$,
%%the variety
$\GG$ is isomorphic to
%%$V/\big({\bf Z}/(r+1){\bf Z}\big)$.
%%$V/\mmu_{r+1}$.
$V/\mmu_m$. Let ${\mathcal I}_{r, m}$ be the set of all
 indecomposable elements
of the additive monoid
%%Let ${\mathcal I}_r$ be the set of
 %%indecomposable elements
%%of the additive monoid
\begin{equation*}\label{monoid}
%%{\mathcal M}_r:=
\biggl\{(a_1,\ldots, a_r)\in {\bf N}^r\;\Big|\;\sum _{i=1}^{r}i a_i\equiv 0\;{\rm mod}\,m\biggl\}.
\end{equation*}
%%and let ${\mathcal I}_r$ be the set of
%%all
%%its
 %%indecomposable elements
%%of ${\mathcal M}_r$
%%i.e., the elements that are not sums of  two nonzero elements of %%${\mathcal M}_r$).
%%this monoid.
Then $k[V/\mmu_{m}$] is isomorphic to the subalgebra of $k[y_1,\ldots, y_r]$ generated by all the monomials
$y_1^{a_1}\cdots y_r^{a_r}$ with $(a_1,\ldots, a_r)\in {\mathcal I}_{r, m}$, and the following equality holds:
%%, for
%%the
%%group \eqref{pgl}, we have
%%$G$
%%we have
%%If ${\mathcal I}_r$ is the set of all
%% indecomposable elements
%%of  ${\mathcal M}_r$, and
%%monoid \eqref{monoid}).
\begin{equation*}
\mathcal H=\biggl\{\sum_{i=1}^{r}a_i\varpi_i\;\Big|\:(a_1,\ldots,a_r)\in {\mathcal I}_{r, m}\biggr\}.
\end{equation*}

For instance, let $r=2$. If ${\rm char}\, k=3$, then $\widehat C=\{e\}$, and
if ${\rm char}\, k\neq 3$,  then the order of   $\widehat C$ is $3$.
%%and in this
In the latter case
%%if $G=\tG/{\widehat
%%C}={\bf PGL}_3$, then
${\mathcal H}=\{3\varpi_1, \varpi_1+\varpi_2, 3\varpi_2\}$ for
$G={\bf PGL}_3$, and $\GG$ is isomorphic to
the surface
$\{(c_1,c_2,c_3)\in {\bf A}^3\mid c_1c_2=c_3^3\}$.

To illustrate the dependence of $\mathcal H$ on ${\rm char}\, k$,  consider the case $r=5$.  Then $\widehat C\neq\{e\}$ and, for
$G={\bf PGL}_6$, the following holds:
%%Let $G=\tG/{\widehat
%%C}={\bf PGL}_7$.
If ${\rm char}\, k\neq 2$, $3$, then $(a,0,0,0,0)\in \mathcal H$ only for $a=6$, but if ${\rm char}\, k=2$ or $3$, then $(6/({\rm char}\, k),0,0,0,0,)\in \mathcal H$. %%$(2,0,0,0,0)\notin \mathcal H$, $(3,0,0,0,0)\notin \mathcal H$.
%%in ${\bf A}^3$
%%defined in the standard coordinates by the equation %%$x_1x_2=x_3^3$.

%%Now assume that $p>0$  and $r+1=p^dm$ where $d>0$ and
%%$p$ does not divide~$m$. Then $t$ in Table 1, case ${\sf A}_r$ %%runs through $m$ different $m$th roots of unity. Fix a generator %%$g$ of $\mmu_m$ and a primitive $m$th root of unity $\zeta$ in %%$k$. Let $L$ be a one-dimensional $\mmu_m$-module on which %%$\mmu_m$ acts by means of the character $g^h\mapsto \zeta^h$. %%Let $V=\bigoplus_{i=1}^{r}L^{\otimes i}$.

%%Let $V$ be
%%the sum of all nontrivial simple modules of the cyclic group ${\bf %%Z}/(r+1){\bf Z}$, each taken with multiplicity $1$ (i.e.,

Note that $|{\mathcal H}|$ ($=\max_{x\in\GG}\dim\,{\rm T}_{x, \GG}$,
%%(equal to the dimension of the tangent space
%%of $\GG$ at the ``vertex'' $v_0$,
see Corollary \ref{TfGG}) grows very rapidly when $r\to\infty$. Indeed,
%%let, for instance, be $m=r+1$. Then
a simple observation from \cite[p.\;105]{Kac}
shows that $|{\mathcal I}_{r, r+1}|\geqslant p(r+1)+\varphi(r+1)-1$ where $p$ and $\varphi$ are, respectively, the classical partition function and the Euler function, and, as is known,
$p(s)\sim %%(1/4s\sqrt{3})\big{\rm exp}\,
\big({\rm exp}\,(\pi\sqrt{2s/3})\big)/4s\sqrt{3}$ when $s\to+\infty$, see \cite{HR}.
%%and $cs/{\rm ln}\,{\rm ln}\,s<\varphi(s)\leqslant s$.

\vskip 1.5mm

(3) Let $\tG$ be of type ${\sf B}_r$  (where ${\sf
B}_1\!:=\!{\sf A}_1$).
Table 1
%%readily
imp\-lies that $\nu^{-1}(\tC)$ is generated by
%%the element
$(1,\ldots,1,-1)$. Hence
${\widehat
C}\neq \{e\}$
%%$Z$
%%is nontrivial
%%(and then coincides with ${\widehat
%%C}$)
(and then $|{\widehat
C}|=2$)
if and only if ${\rm char}\,k\!\neq\!2$.
Assume that this inequality holds.
As $1\neq -1$, we then have $k[{\bf A}\!^r]^{\widehat
C}\!\!=\!k[y_1,\ldots, y_{r-1}, y_r^2]$. Therefore,
%%for
%%the adjoint $G$, i.e.,
for $G=\tG/{\widehat
C}={\bf SO}_{2r+1}$, we have
\vskip 1mm
$$\mathcal H=\{\varpi_1,\ldots,\varpi_{r-1}, 2\varpi_{r}
\}
$$
\vskip 2mm
\noindent and
%%the
%%variety
$\GG$ is isomorphic to ${\bf A}\!^r$ (the latter
agrees with Theorem  \ref{sing}).

\vskip 1.5mm

(4) Let $\tG$ be of type ${\sf D}_r$, $r\geqslant 3$, and
%%Let ${\rm char}\,k\neq 2$ and
let $Z:=\{t\in \widehat
C\mid t^{\varpi_1}=1\}$. Table 1 implies that $\nu^{-1}(Z)$ is generated by
%%the element
$(1,\ldots, 1, -1, -1)$. Hence $Z\neq \{e\}$
%%is nontrivial
if and only if ${\rm char}\,k\neq 2$. Assume that this inequality holds. As $1\neq -1$, we then have
%%Then $G:=\tG/Z={\bf SO}_{2r}$.
%% Table 1 implies that $\nu^{-1}(Z)$ is generated by
%%$(1,\ldots,1,-1,-1)$. Whence
$k[{\bf A}\!^r]^Z=k[y_1,\ldots,y_{r-2},
y_{r-1}^2, y_{r}^2, y_{r-1}y_{r}]$.
Therefore, for
%%the group
$G:=\tG/Z={\bf SO}_{2r}$, we have
\vskip .3mm
$$\mathcal H=\{\varpi_1,\ldots,\varpi_{r-2}, 2\varpi_{r-1}, 2\varpi_{r}, \varpi_{r-1}+ \varpi_{r}
\}
$$
\vskip 2mm
\noindent and
%%the variety
$\GG$ is isomorphic
to ${\bf A}^{r-2}\!\times X$ where $X$
is a nondegenerate quadratic cone in\;${\bf A}^3$.
%% and

\vskip 1.5mm

(5) Let $\tG$ be of type ${\sf D}_r$ with even
$r=2d\!\geqslant\! 4$
%%Let ${\rm char}\,k\!\neq\! 2$
and let $Z\!:=\!\{t\!\in\! \widehat C\mid
t^{\varpi_r}\!=\!1\}$.
%%Then $G:=\tG/Z$ is the
%%half-spinor group ${\bf Spin}_{2r}^{1/2}$.
Table 1
implies that $\nu^{-1}(Z)$ is generated by $(-1, 1, -1,
1, \ldots, -1, 1)$. Hence $Z\neq \{e\}$ if and only if ${\rm char}\,k\neq 2$. Assume that this inequality holds. As $1\neq -1$,
the algebra
%%we then have that
%%
%%
%%Whence
$k[{\bf A}\!^r]^Z$ is then minimally
generated by all $y_i$'s with even $i$ and all the
monomials of degree $2$ in $y_j$'s with odd $j$.
Therefore, for
%%the group
$G:=\tG/Z={\bf Spin}_{2r}^{1/2}$
%%, that is for
(the
half-spinor group),
%%${\bf Spin}_{2r}^{1/2}$,
we have
\vskip 1mm
$$
\mathcal H=\{\varpi_i\mid i\hskip 2mm\mbox{is
even}\}\cup\{\varpi_l+\varpi_m\mid l, m \hskip
2mm\mbox{are odd}\}
$$
\vskip 2mm\noindent
and
%%the variety
$\GG$ is isomorphic to ${\bf A}^d\times Y$
where $Y$ is the affine cone over the Veronese variety
$\nu_2({\bf P}^{d-1})$ in ${\bf P}^{(d-1)(d+2)/2}$.
Note that if $r=4$, then, up to the renumbering of simple roots, we obtain the same answer as in the previous example.

\vskip 1.5mm

(6) Let $\tG$ be of type ${\sf E}_7$.
%% and let ${\rm
%%char}\,k\neq 2$.
Table 1 implies that  $\nu^{-1}(\tC)$ is generated by
the element
$(1,-1,1,1,-1,1,-1)$. Hence
${\widehat
C}\neq \{e\}$
%%$Z$
%%is nontrivial
(and then $|{\widehat
C}|=2$)
%%coincides with ${\widehat
%%C}$)
if and only if ${\rm char}\,k\!\neq\!2$.
Assume that this inequality holds.
Then, as $1\neq -1$, the algebra
$k[{\bf
A}\!^7]^{\widehat C}$ is minimally generated by $y_1,
y_3, y_4, y_6$ and all the monomials of degree $2$ in $y_2,
y_5, y_7$. Therefore, for
%%the adjoint $G$, i.e., for
$G=\tG/{\widehat
C}$, we have\vskip 1mm
$$
\mathcal H=\{\varpi_1, \varpi_3, \varpi_4, \varpi_6,
2\varpi_{2}, 2\varpi_{5}, 2\varpi_{7},
\varpi_{2}+\varpi_{5}, \varpi_{2}+\varpi_{7},
\varpi_{5}+\varpi_{7}\}
$$
\vskip 2mm
\noindent
and
%%the variety
$\GG$ is
isomorphic to ${\bf A}^4\times Y$ where $Y$ is the
affine cone over the Veronese variety $\nu_2({\bf
P}^2)$ in ${\bf P}^5$ (in particular, the maximum of dimensions of tangent spaces
of the $7$-dimensional variety $\GG$ is $10$.)
%% at the
%%unique
%%fixed point $v_0$ of $T$
%%, see
%%Theorem \ref{toric}(i) and
%%\eqref{tange},
%%``vertex'' $v_0$, see Corollary \ref{TfGG},
%%is $10$).
%%-dimensional).
\qquad $\square$

\end{examples}

\begin{remark}\label{smo} {\rm Considering in the same way the remaining types of simple groups  one obtains the proof of
Theorem \ref{sing}.}
\end{remark}

\section{Two further questions of Grothendieck}\label{Yettwo}

Theorem \ref{generators} describes a minimal generating
set of the algebra $k[G]^G$ of class functions on $G$.
Constructing generating sets of $k[G]^G$ is the topic
of two further questions of {\sc Grothendieck} in
\cite[p.\;241]{GS}:
\begin{quote} ``$[\ldots]$ {\it When $G$ is an adjoint group, is it
possible to generate the affine ring of $I(G)$ with
coefficients of the Killing polynomial{\rm?} In the
general case, is it enough to take the coefficients of
analogous polynomials for certain linear representations
{\rm(}perhaps arbitrary faithful representations})?
$[\ldots]$''
\end{quote}
Below we answer these questions.

Let $\varrho\colon G\!\to\! {\bf GL}(V)$ be a finite
dimensional linear representation of $G$. Define the~set\vskip 1mm
\begin{equation*}
C_{\varrho}:=\{c_{\varrho, i}\in k[G]\mid i=1,\ldots,
\dim V\}
\end{equation*}
\vskip 2mm\noindent
by the equality
\begin{equation}\label{charpoly}
{\rm det}(xI-\varrho(g))=\sum_{i=0}^{\dim V}
c_{\varrho, i}(g)x^{\dim V-i}\quad\mbox{for every}\hskip
2mm g\in G,
\end{equation}
where $x$ is a variable.
 If
$V\!=\!E(\varpi)$ (here and below we use the notation
of Section \ref{singugen}) and $\varrho$ determines the
$G$-module structure of $E(\varpi)$, we put
$C_{\varpi}:=C_{\varrho}$.

Clearly, $c_{\varrho, i}\!\in\! k[G]^G$ and $c_{\varrho,
1}$ is the character of $\varrho$. Hence by Theorem
\ref{generators}(iii)
\begin{equation*}\label{Gro2}
\bigcup_{\varpi\in \mathcal H} C_{\varpi}
\end{equation*}
is a generating set of the algebra $k[G]^G$. This
answers  the second {\sc Gro\-then\-dieck}'s question in
the affirmative.

In order to answer the first one in the negative it is
sufficient to find an adjoint $G$ and two elements $z_1,
z_2\in T$ such that
\begin{enumerate}
\item[(i)] $z_1$ and $z_2$ are not in the same
$W$-orbit; \item[(ii)] the spectra
of the linear operators ${\rm Ad}_{G}\,z_1$ and ${\rm
Ad}_{G}\,z_2$ on the vector space ${\rm Lie}\,G$
coincide.
\end{enumerate}
Indeed, property (i) implies that there is a function
$f\in k[T]^W$ such that $f(z_1)\neq f(z_2)$. Given
isomorphism \eqref{G//G}, this means that there is a
function ${\widetilde f}\in k[G]^G$ such that
${\widetilde f}(z_1)\neq {\widetilde f}(z_2)$. On the
other hand, \eqref{charpoly} and property (ii) imply
that\vskip 1mm
\begin{equation*}
c_{{\rm Ad}_G,i}(z_1)=c_{{\rm Ad}_G,i}(z_2)\qquad
\mbox{for every}\hskip 2mm i.
\end{equation*}\vskip 2mm\noindent
Therefore, $\widetilde f$ is not in the subalgebra of
$k[G]^G$ generated by $C_{\rm Ad_G}$, i.e., the latter
is not a generating set of $k[G]^G$.

The following two examples show that one indeed can find
$G$, $z_1$, and $z_2$ sharing properties (i) and (ii).
\begin{examples}\

(1) Let $G=H\times H$ where $H$ is a connected adjoint
semisimple algebraic group. Let $T=S\times S$ where $S$
is a maximal torus of $H$. Let $W_S$ be the Weyl group
of $H$ naturally acting
%%naturally
on $S$. Take any two elements
$a, b\in S$ that are not in the same $W_S$-orbit and put
$z_1:=(a,b),\; z_2:=(b, a)\in T$. As $W=W_S\times W_S$,
property (i) holds. On the other hand,
clearly, for every $i=1, 2$, the spectrum of ${\rm
Ad}_G\,z_i$ is the union of the spectra of ${\rm
Ad}_H\,a$ and ${\rm Ad}_H\,b$; whence property (ii)
holds.

\vskip 1mm

(2) In this example $G$ is simple, namely, $G={\bf
PGL}_3$. Let $\alpha_1, \alpha_2\in {\rm X}(T)$ be the
simple roots of $T$ regarding  $B$. As the map
$T\to {\bf G}_{\rm m}^2$, $t\mapsto (t^{\alpha_1},
t^{\alpha_2})$, is surjective (in fact, an isomorphism),
for every $u, v\in k$, $uv\neq 0$, there are $z_1,
z_2\in T$ such that $z_1^{\alpha_1}=u$,
$z_1^{\alpha_2}=v$ and $z_2^{\alpha_1}=v$,
$z_2^{\alpha_2}=u$. For  these $z_1, z_2$, property (ii)
holds as the set of roots of $G$ regarding  $T$ is
$\{\pm\alpha_1, \pm\alpha_2, \pm(\alpha_1+\alpha_2)\}$.
Now take $u$ and $v$ such that
all the elements
$u$, $u^{-1}$, $v$, $v^{-1}$, $uv$, $u^{-1}v^{-1}$ are
pairwise
%%different
distinct. Then property (i) holds as there are no $w\in
W$ such that $w(\alpha_1)=\alpha_2$ and
$w(\alpha_2)=\alpha_1$.
\end{examples}

\section{Rational cross-sections }\label{si}\label{rcs}

Recall from \cite[2.14, 2.15]{St1} that an element $x\in
G$ is called {\it strongly regular} if its centralizer
$G_x$ is a maximal torus. Such $x$ is regular and
semisimple. Strongly regular
elements form a dense open subset $G_0$ of $G$
stable with respect to the conjugating action of $G$.
Every $G$-orbit in $G_0$ is regular and closed in $G$.
We put\vskip 1mm
\begin{equation*}\label{G0T0}
(\GG)_0:=\pG(G_0)\qquad\mbox{and}\qquad  T_0:=T\cap G_0.
\end{equation*}
\vskip 2mm\noindent
Abusing the notation, we
denote
$\pG|_{G_0}$ still by $\pG$:
\begin{equation}\label{pG0000}
\pG\colon G_0\longrightarrow (\GG)_0.
\end{equation}
\begin{lemma}\label{GG0}
\

\begin{enumerate}
\item[\rm(i)] $(\GG)_0$ is an open
smooth subset of $\GG$. \item[\rm(ii)] $\pG|_{T_0}\colon
T_0\to (\GG)_0$ is a surjective \'etale map.
\item[\rm(iii)] $((\GG)_0, \pG)$ is the geometric
quotient for the action of $\;G$ on $G_0$.
\end{enumerate}
\end{lemma}
\begin{proof} Since $\GG$ is normal and all
fibers of $\pG$ are of constant dimension and
irreducible,
$\pG$ is an open map (see \cite[AG.18.4]{Bor}). Hence
$(\GG)_0$ is open in $\GG$.

As every element of $G_0$ is semisimple, it is conjugate
to an element of $T_0$; whence $\pG|_{T_0}$ is
surjective.

The set
$T_0$
is open in $T$ and $W$-stable. For every point $t\in
T_0$, we have $G_t=T$, hence the $W$-stabilizer of $t$
is trivial. Thus, the action of $W$ on $T_0$ is set
theoretically free.
Since $T$ is
smooth, $\GG$ is normal, and $(\GG, \pG|_T)$ is the
quotient for the action of $W$ on $T$ (see
\cite[6.4]{St1}),
we deduce from \cite[Exp.\;I, Th\'eor\`eme 9.5(ii)]{G3}
and \cite[V.2.3, Cor.\,4]{Bour} that $\pG|_{T_0}$ is
\'etale and hence
$(\GG)_0$ is
smooth.
This proves (i) and (ii).

By (ii) the map $\pG\colon G_0\to (\GG)_0$ is separable
and surjective. As its fibers are $G$-orbits and
$(\GG)_0$ is normal, (iii) follows from \cite[6.6]{Bor}.
\quad $\square$
\renewcommand{\qed}{}\end{proof}

The group $W$ acts on $G/T\times T_0$  diagonally with
the action  on the first factor defined by formula
\eqref{G/T}. The group $G$ acts on $G/T\times T_0$  via
left translations of the first factor. These two actions
commute with each other.

Consider the $G$-equivariant morphism
\begin{equation}\label{psss}
\psi\colon G/T\times T_0\longrightarrow G_0,\qquad (gT,
t)\mapsto gtg^{-1}.
\end{equation}

The proofs of Lemma \ref{dpsi} and Corollary
\ref{vartheta} reproduce that from my letter \cite{P2}.

\begin{lemma}\label{dpsi}
$\psi$ is a surjective \'etale map.
\end{lemma}
\begin{proof}
As every $G$-orbit in $G_0$ intersects $T_0$,
surjectivity of $\psi$ follows from \eqref{psss}.

Take a point $z\in G/T\times T_0$. We shall prove that $d\psi_z$ is an isomorphism. As $G/T\times T_0$ and $G_0$ are
smooth, this is equivalent to proving that $\psi$ is
\'etale at $z$. Using that $\psi$ is $G$-equivariant, we
may assume that $z=(eT, s)$, $s\in T_0$.

Let $U_\alpha$ be the one-dimensional unipotent root
subgroup of $G$ corresponding to a root $\alpha$ with
respect to $T$ and let
$\theta_\alpha\colon {\bf G}_a\!\to\! U_\alpha$
be the isomorphism of groups such that\vskip 1mm
\begin{equation*}
t\theta_\alpha(x)t^{-1}=\theta_\alpha(t^{\alpha}x)
\qquad\mbox{ for all $\ t\in T, \ x\in {\bf G}_a$},
\end{equation*}\vskip 2mm\noindent see \cite[IV.13.18]{Bor}. Put
\begin{gather*}
C_\alpha:=\{(\theta_\alpha(x)T, s)\in G/T\times T_0\mid x\in {\bf G}_a\},\\
D:=\{(eT, t)\in G/T\times T_0\mid t\in T_0\}.
\end{gather*}

The linear span of all tangent spaces ${\rm T}_{z, C_{\alpha}}$
%%\hskip
%%-1mm's
and
${\rm T}_{z, D}$ is ${\rm T}_{z, G/T\times T_{0}}$.
We have
\begin{equation}\label{alllg}
\begin{split}
\psi(\theta_\alpha(x)T,
s)&=\theta_\alpha(x)s\theta_\alpha(x)^{-1}=
\theta_\alpha(x)s\theta_\alpha(-x)\\
&=\theta_\alpha(x)\theta_\alpha(-s^\alpha x)s=
\theta_\alpha((1-s^\alpha)x)s.
\end{split}
\end{equation}
Since $s$ is
regular, $s^\alpha\neq 1$.
Hence \eqref{alllg} shows that $\psi$ maps the curve
$C_\alpha$ isomorphically onto the curve\vskip 1mm
\begin{equation*}
\psi(C_\alpha)=\{\theta_\alpha((1-s^\alpha)x)s\mid x\in
{\bf G}_\alpha\}.
\end{equation*}
\vskip 2mm

Clearly, $\psi(D)=T_0$ and
$\psi|_D\colon D\to T_0$
is an isomorphism.
But ${\rm T}_{e, G}$ is the linear span of ${\rm T}_{e,
T}$ and the tangent spaces to the curves
$\{\theta_\alpha(x)\mid x\in {\bf G}_\alpha\}$
at $e$.
Hence ${\rm T}_{s, G}$ is the linear span of ${\rm
T}_{s, T}$ and the tangent spaces  at $s$ to the right
translates of these curves by $s$.
This implies the claim of the lemma.
\quad $\square$
\renewcommand{\qed}{}\end{proof}
\begin{corollary} \label{psisep}
$\psi$ is separable.
\end{corollary}

\begin{corollary}\label{vartheta}
$(G_0, \psi)$ is the quotient for the action of $\;W$ on
$G/T\times T_0$.
\end{corollary}

\begin{proof}
By \cite[Prop.\,II.6.6]{Bor}, as $G_0$ is normal and
$\psi$ is surjective and separable, it suffices to prove
that the fibers of $\psi$ are $W$-orbits.

%%%%%%%%%%%%%%%%%%%%%%%%%%%%%%%%%%

Using \eqref{G/T} and \eqref{psss} one immediately
verifies that the fibers of $\psi$ are $W$-stable. On
the other hand, let $\psi(g_1T, t_1)\!=\!\psi(g_2T,
t_2)$. By \eqref{psss} this equality is equivalent to
$(g_1^{-1}g_2)t_2(g_1^{-1}g_2)^{-1} \!\!=t_1$. By
\cite[6.1]{St1} the latter, in turn, implies that there
is an element $w\in W$ such that\vskip 1mm
\begin{equation*}\label{wg}
\overset{.}{w}t_2\overset{.}{w}^{-1}=
(g_1^{-1}g_2)t_2(g_1^{-1}g_2)^{-1}.
\end{equation*}
\vskip 2mm\noindent
Hence $g_1^{-1}g_2=\overset{.}{w}z$ for $z\in G_{t_2}$.
As $t_2\in T$ is strongly regular, this yields that
$z\in T$.
Therefore,\vskip 1mm
\begin{equation*}
(g_2T, t_2)=(g_1\overset{.}{w}T,
\overset{.}{w}^{-1}t_1\overset{.}{w})= w^{-1}\cdot(g_1T,
t_1).
\end{equation*}\vskip 2mm\noindent
Thus, $(g_1T, t_1)$ and $(g_2T, t_2)$ are in the same
$W$-orbit.  This completes the proof. \quad $\square$
\renewcommand{\qed}{}\end{proof}

Let $\pi^{}_2\colon G/T\times T_0\to T_0$ be the second
projection.
Clearly, $(T_0, \pi^{}_2)$ is the
geometric quotient for the action of $G$ on $G/T\times
T_0$. As $\psi$ is $G$-equivariant, this implies that
there is a morphism $\phi\colon T_0\to \GG$ such that
the following diagram is commutative:\vskip 1mm
\begin{equation}\label{cddd}
\begin{matrix}
\xymatrix{G/T\times T_0\ar[d]_{\pi^{}_2}\ar[r]^{\hskip 4mm\psi} & G_0\ar[d]^{\pG}\\
T_0\ar[r]^{\phi}&(\GG)_0 }
\end{matrix}\quad .
\end{equation}\vskip 2mm

\begin{lemma}\label{commutdia}
\
\begin{enumerate}
\item[\rm(i)] $\phi=\pG|_{T_{0}}$. \item[\rm(ii)]
For every point $t\in T_0$, the restriction of $\psi$ to
$\pi_2^{-1}(t)$ is a $G$-equivariant isomorphism
$\pi_2^{-1}(t)\to \pi_G^{-1}(\phi(t))$.
\end{enumerate}
\end{lemma}
\begin{proof} Take a point $t\in T_0$. Commutativity of diagram \eqref{cddd} and formula \eqref{psss} yield that
$\pG(t)=\pG(\psi(eT, t))=\phi(\pi^{}_2(eT, t))=
\phi(t)$. This proves (i).

Commutativity of diagram \eqref{cddd} implies that the
restriction of $\psi$ to $\pi_2^{-1}(t)$  is a
$G$-equivariant morphism $\pi_2^{-1}(t)\to
\pi_G^{-1}(\phi(t))$. As both
$\pi_2^{-1}(t)$ and  $\pi_G^{-1}(\phi(t))$ are
$G$-orbits
and the stabilizers of their points are conjugate to
$T$, this morphism is bijective. By Lemma \ref{dpsi} it
is separable. Then, as $\pi_G^{-1}(\phi(t))$ is normal,
it is an isomorphism. This proves (ii). \quad $\square$
\renewcommand{\qed}{}\end{proof}

\begin{proof}[Proof of Theorem {\rm \ref{rsrcs}}]
Assume that (i) holds. Let $\sigma\colon
\GG\dashrightarrow G$ be a rational section of $\pG$,
i.e., a section of $\pG$ over a dense open subset $U$ of
$(\GG)_0$. Let $S$ be the closure of $\sigma(U)$. Put
$\rho:=\pG|_S\colon S\to (\GG)_0$. Since
$\pG\circ\sigma={\rm id}$, shrinking $U$ if necessary,
we may
assume that, for every point $x\in U$, the following
properties hold:
\begin{enumerate}
\item[(a)] $S\cap \pi_G^{-1}(x)$ is a single point $s$;
\item[(b)] $d\rho_s$ is an isomorphism.
\end{enumerate}

Since $\psi$ is an isomorphism on the fibers of $\pi_2$,
property (a) implies that, for every point $t\in
\phi^{-1}(U)$, the $W$-stable closed set $\psi^{-1}(S)$
intersects $\pi_2^{-1}(t)$ at a single point. From this
we infer that $\psi^{-1}(S)$ has a unique irreducible
component $\widetilde S$ whose image under $\pi_2$ is
dense in $T_0$\,---\,the argument is the same as that in
the proof of Claim 2(i) in Section~\ref{crssect}. Due to
the uniqueness, this $\widetilde S$ is $W$-stable.

Let $V\subseteq \pi^{}_2(\widetilde S)\cap \phi^{-1}(U)$
be an open subset of $T_0$. Replacing it, if necessary,
by $\bigcap_{w\in W} w(V)$, we may assume that $V$ is
$W$-stable. Let ${\widetilde\rho}\colon
\pi_2^{-1}(V)\cap \widetilde S\to V$ be the restriction
of $\pi_2$ to $\pi_2^{-1}(V)\cap \widetilde S$. Then
$\widetilde\rho$ is a bijective $W$-equivariant
morphism. We claim that it is separable and hence, by
{\sc Zariski}'s Main Theorem, an isomorphism  (as $V$ is
normal). Indeed, take a point $\widetilde s\in
\pi_2^{-1}(V)\cap \widetilde S$ and put
$\pi^{}_2(\widetilde s)=t$. Then property (b), Lemma
\ref{dpsi}, and commutativity of diagram \eqref{cddd}
imply that $d{\widetilde\rho}_s\colon {\rm
T}_{\widetilde s, \widetilde S}\to {\rm T}_{t, V}$ is an
isomorphism; whence the claim by \cite[AG.17.3]{Bor}.

Thus, ${\widetilde\rho}^{-1}\colon
V\to  \pi_2^{-1}(V)\cap \widetilde S$ is a rational
$W$-equivariant section of $\pi_2$. Its composition with
the first projection $G/T\times T_0\to G/T$ is then a
$W$-equivariant rational map $T\dashrightarrow G/T$.
This proves (i)$\Rightarrow$(ii).

\vskip 1mm

Conversely, assume that (ii) holds. Let $\gamma\colon
T\dashrightarrow G/T$ be a $W$-equivariant rational map.
Then $\varsigma:=(\gamma, {\rm id})\colon
T_0\dashrightarrow G/T\times T_0$ is a $W$-equivariant
rational section of $\pi^{}_2$, i.e., a section of
$\pi^{}_2$ over a dense open subset $V$ of $T_0$. We may
assume that $\varsigma(V)$ and $S:=\psi(\varsigma(V))$
are open in their closures, $\varsigma\colon V\to
\varsigma(V)$ is an isomorphism, and the subsets
$\phi(V)$, $\pG(S)$ of $\GG$ are open and coincide.
As above, we may also assume that $V$ is $W$-stable.

Taking into account that $\varsigma$ is $W$-equivariant,
$\varsigma(V)\cap \pi_2^{-1}(t)$ is a single point for
every $t\in V$, and $\psi$ is an isomorphism on the
fibers of $\pi^{}_2$, we conclude that
property (a) holds for every $x\in \varsigma(V)$. Thus,
$\rho:=\pG|_S\colon S\to \phi(V)$ is a bijection.

We claim that $\rho$ is separable, hence an isomorphism
as $\phi(V)$ is normal by Lemma \ref{GG0}(i).
Indeed, $d\phi_t$ is an isomorphism by Lemma
\ref{commutdia}(i) and Lemma \ref{GG0}(ii).  Let
$s=\psi(\varsigma(t))\in S$. %%
Since the restriction of $(d\pi^{}_2)_{\varsigma(t)}$ to
${\rm T}_{\varsigma(t), \varsigma(V)}$ is an isomorphism
with ${\rm T}_{t, V}$, commutativity of diagram
\eqref{cddd} and
Lemma \ref{dpsi} imply that
property (b) holds; whence the claim.

Thus, the composition of $\rho^{-1}\colon \phi(V)\to S$
and the identical embedding $S\hookrightarrow G$ is a
rational section of $\pG$. This proves
(ii)$\Rightarrow$(i)
 and completes the proof of the theorem.~\quad $\square$
\renewcommand{\qed}{}\end{proof}

\vskip -1mm

Recall some definitions from \cite[Sects.\;2.2, 2.3, and
3]{CTKPR}.

Let $P$ be a linear algebraic group acting on a variety
$X$ and let $Q$ be its closed subgroup. $X$ is called a
$(P, Q)$-{\it variety} if in $X$ there is a dense open
$P$-stable subset $U$, called {\it a friendly subset},
such that a geometric quotient $\pi^{}_U\colon U\to U/P$
exists and $\pi^{}_U$
becomes the second projection $P/Q\times \widehat{U/P}\to \widehat{U/P}$
after a surjective \'etale base change
$\widehat{U/P}\to U/P$. If there is a rational section
of $\pi^{}_U$, one says that $X$ {\it admits a rational
section}. $X$ is called a {\it versal $(P, Q)$-variety}
\;if $U/P$ is irreducible and, for every its dense open
subset $(U/P)_0$ and $(P, Q)$-variety $Y$, there is a
friendly subset $V$ of $Y$ such that $\pi^{}_V$
is induced from $\pi^{}_U$
by a base change $V\to (U/H)_0$.

\vskip 1.5mm

Now we shall give the characteristic free proofs of the
following two statements proved  in \cite{CTKPR} for
${\rm char}=0$.

\begin{lemma}\label{fingrou}
Let $X$ be an irreducible  variety endowed with a
faithful action of a finite algebraic group $H$. Then
\begin{enumerate}
\item[\rm(i)] $X$ is an $(H,\{e\})$-variety;
\item[\rm(ii)] $X$ is a versal $(H,\{e\})$-variety in
each of the following cases:
\begin{enumerate}
\item[\rm(a)] $X$ is a free $H$-module; \item[\rm(b)]
$X$ is a linear algebraic torus and $H$ acts by its
automorphisms.
\end{enumerate}
\end{enumerate}
\end{lemma}
\begin{proof}
 (i) Replacing $X$ by its smooth locus, we may assume that $X$ is smooth.

As $H$ is finite, for any nonempty open affine subset
$U$ of $X$, the set   $\bigcap_{h \in H}h(U)$ is
$H$-stable, affine, and open in $X$. So, replacing $X$
by it, we may assume that $X$ is affine. Then, as is
well known, for the action of $H$ on $X$ there is a
geometric quotient $\pi\colon X\to X/H$ (see,
e.g.,\;\cite[Prop.\;6.15]{Bor}).  As $X$ is normal,
$X/H$ is normal as well.

Since $H$ is finite and the action is faithful, the
points with trivial stabilizer form an open subset of
$X$. Replacing $X$ by it, we may also assume that the
action is set-theoretically free, i.e., the
$H$-stabilizer of every point of $X$ is trivial. As $X$
and $X/G$ are normal, by \cite[Exp.\;I, Th\'eor\`eme
9.5(ii)]{G3} and \cite[V.2.3, Cor.\,4]{Bour} this
property implies  that $\pi$ is \'etale and hence $X/H$
is smooth.

For every base change $\beta\colon Y\to X/H$ of $\pi$,
the group $H$ acts on $X\times_{X/H} Y$ via $X$. As the
action of $H$ on $X$ is set-theoretically free, taking
$Y=X$ and $\beta=\pi$, we obtain\vskip 1mm
\begin{equation*}\label{XXX}
X\times_{X/H}X=\bigsqcup_{h\in H}
h(D)\qquad\mbox{where}\quad D:=\{(x,x)\mid x\in X\}.
\end{equation*}
\vskip 2mm\noindent
From this we deduce that  in the commutative diagram

\begin{equation*}\label{trianggg}
\quad\begin{matrix}
\xymatrix@C=3mm@R=5mm{H\times X\ar[rr]^{\hskip
-3mm\alpha} \ar[dr]
&&X\times_{X/H}X\ar[dl]\\
&X
& }
\end{matrix}\quad,
\end{equation*}
where $\alpha(h,x):=(h(x), x)$
and two other maps are the second projections, $\alpha$
is an $H$-equivariant isomorphism. This proves (i).

The proofs of (ii)(a) and (ii)(b) are the same as that
of (b) and (d) in \cite[Lemma 3.3]{CTKPR} if one
replaces in them the references to \cite[Theorem
2.12]{CTKPR} (whose proof is based on the assumption
${\rm char}\,k=0$) by the references to statement (i) of
the present lemma. \quad $\square$
\renewcommand{\qed}{}\end{proof}
\begin{remark} {\rm The proof of (i) shows that, for finite group actions, set-theoretical freeness coincides with that in the sense of GIT,
\cite[Def.\;0.8]{MF}.}
\end{remark}

\begin{lemma}\label{G/Tversal}
$G$ is a versal $(G,T)$-variety.
\end{lemma}

\begin{proof} First, we shall give
a characteristic free proof of the fact that $G$ is a
$(G, T)$-variety
(the proof given in \cite{CTKPR} is based on the
assumption ${\rm char}\,k=0$).
By Lemma \ref{GG0}(iii) this is equivalent to proving
the existence of a dense open subset $U$ of $(\GG)_0$
such that after a surjective \'etale base change $U'\to
U$
morphism \eqref{pG0000}
becomes the second projection $G/T\times U'\to U'$.

Consider the base change of $\pG$ in \eqref{cddd} by
means of $\phi$. Lemma \ref{commutdia}(i) implies that\vskip 1mm
\begin{equation}
\label{FP} F:=G_0\times_{(\GG)_0} T_0=\{(g, t)\in
G_0\times T_0 \mid G(g)=G(t)\}
\end{equation}
\vskip 2mm\noindent
(see \eqref{cc}).
We have the canonical map corresponding to commutative
diagram \eqref{cddd}:\vskip 1mm
\begin{equation}\label{cmap}
\gamma:=\psi\times {\rm id}\colon G/T\times
T_0\longrightarrow F,\qquad (gT, t)\mapsto (gtg^{-1},
t).
\end{equation}\vskip 2mm\noindent
It follows from \eqref{FP} that $\gamma$ is surjective;
whence $F$ is irreducible. But if for $t\in T_0$ and
$g_1, g_2\in G$ we have $g_1tg_1^{-1}=g_2tg_2^{-1}$,
then $g_1T=g_2T$ since $G_t=T$. Therefore, $\gamma$ is
bijective. Lemma \ref{dpsi} and \eqref{cmap} show that
$d\gamma_x$ is injective for every $x\in G/T\times T_0$.
Hence if $\gamma(x)$ lies
in  the smooth locus $F_{\rm sm}$ of $F$, then
$d\gamma_x$ is the isomorphism. This implies that
$\gamma$ is separable and then, by {\sc Zariski}'s Main
Theorem, that the restriction of $\gamma$
to $\gamma^{-1}(F_{\rm sm})$
is an isomorphism $\gamma^{-1}(F_{\rm sm})\to F_{\rm
sm}$.

As $F_{\rm sm}$ is $G$-stable and $\gamma$ is
$G$-equivariant,  $\gamma^{-1}(F_{\rm sm})$ is a
$G$-stable open subset of $G/T\times T_0$. Hence
it is of the form $G/T\times U'$ for
an open subset $U'$
of $T_0$. But Lemmas \ref{G0T0}(ii) and
\ref{commutdia}(i) imply that $U:=\phi(U')$ is open in
$(\GG)_0$ and $\phi|_{U'}\colon U'\to U$ is \'etale.
This proves that after the \'etale base change
$\phi|_{U'}\colon U'\to U$ morphism \eqref{pG0000}
becomes the second projection $G/T\times U'\to U'$.
Hence $G$ is a $(G, T)$-variety.

By Lemma \ref{fingrou}(b), $T$ is a versal
$(W,\{e\})$-variety. The characteristic free arguments
from \cite[proof of Prop.\;4.3(c)]{CTKPR} then show that
this fact implies versality of the $(G, T)$-variety $G$.
This completes the proof of the lemma. \quad $\square$
\renewcommand{\qed}{}\end{proof}

\begin{proof}[Proof of
Theorem {\rm \ref{anychar}}]
%%By \cite[Prop.\;(2.24)(ii)]{BT}
Since the isogeny $\tau$ is central,
%%Hence,
the natural morphism
$\tG/\tT\to G/T$ is an isomorphism by
%%\cite[2.17]{BT} and
\cite[Props.\;6.13, 22.5]{Bor}.

Using $\tau$, every action of $G$ naturally lifts
%%naturally
to an
action of $\tG$ on the same variety. In particular, $G$
is endowed with
an action of $\tG$.  But $G$ is a $(G, T)$-variety by
Lemma \ref{G/Tversal}(i).
As $\tG/\tT$ and $G/T$ are isomorphic, this means that
$G$ is a $(\tG,\tT)$-variety. But $\tG$ is a versal
$(\tG,\tT)$-variety (by Lemma \ref{G/Tversal}) that
admits a rational section (by Lemma \ref{GG0}(iii) and
\cite[Theorem 1.4]{St1}). Hence by \cite[Theorem 3.6(a)]{CTKPR} (the
proof of this result is characteristic free) every
$(\tG, \tT)$-variety admits a rational section. In
particular, this is so for $G$.  This proves (ii) and
completes the proof of the theorem. \quad $\square$
\renewcommand{\qed}{}\end{proof}

\begin{proof}[Proof of
Corollary {\rm \ref{rtcrsct}}] By Theorem \ref{anychar} there is a rational section $\sigma\colon \GG\dashrightarrow G$ of $\pG$. The closure of the image of $\sigma$ is then the desired cross-section $S$ (see Subsection\;\ref{re}.A).
\quad $\square$\renewcommand{\qed}{}\end{proof}

\section{Complements} \label{re}
\subsection*{6.A.\;Cross-sections versus sections} \label{re1}
If there is a section $\sigma\colon \GG\to G$ of $\pG$,
then $\sigma(\GG)$ is a cross-section in $G$. Indeed, as
${\rm id}_{k[\GG]}$ is the composition of the
homomorphisms\vskip 1mm
$$ k[\GG]\xrightarrow{\pi_G^*} k[G]\xrightarrow{\sigma^*} k[\GG],
$$\vskip 2mm\noindent
$\pi_G^*$ is surjective; by \cite[Cor.\;4.2.3]{G2} this
means that $\sigma$ is
a closed embedding.

The cross-section $\sigma(\GG)$ has the property that
the restriction of $\pG$
to $\sigma(\GG)$ is an isomorphism
$\sigma(\GG)\to \GG$.
Conversely, let $S$ be a cross-section in $G$. If
$\pG|_S\colon S\to \GG$
is separable, then, since
$\pG|_S$ is bijective and $\GG$ is normal, {\sc
Zariski}'s Main Theorem implies that $\pG|_S$ is an
isomorphism
(cf.\;\cite[AG\,18.2]{Bor}). So in this case the
composition of $(\pG|_S)^{-1}$ with the identity
embedding $S\hookrightarrow G$ is a section of $\pG$
whose image is $S$. In particular, if ${\rm char}\,k=0$,
then every cross-section in $G$ is the image of a
section of $\pG$. If ${\rm char}\,k>0$,
then in the general case this
is not true.

\begin{example}\label{exm}  Let $G={\bf SL}_3$ and ${\rm char}\,k=p>0$. Then for every integer $d>0$,
\begin{equation*}
S:=\{s(a_1, a_2)\mid a_1, a_2\in k\}, \qquad
\mbox{where}\quad
s(a_1, a_2):=
\begin{pmatrix}a_1 & a_2 & 1\\
1& a^{p^d}_1-a_1 &0\\
0& 1 & 0
\end{pmatrix},
\end{equation*}
\noindent is a cross-section in $G$ such that
$\pG|_S$
 is
not separable. Indeed, as ${\rm ch}_{\varpi_i}(g)$ is the
sum of principal $i$-minors of $g\in G$, we have (see
Lemma \ref{lin}(ii))\vskip 1mm
$$(\lambda\circ\rho)\big(s(a_1, a_2)\big)=\big(a^{p^d}_1, a_1\big(a_1^{p^d}-a_1\big)-a_2\big).\qquad \square$$
\end{example}
\vskip 2mm

Similarly, if $\sigma\colon \GG\dashrightarrow G$ is a
rational section of $\pG$ and $S$  is the closure of
its image, then $S$ is a rational cross-section in $G$
such that the restriction of $\pG$ to it is a
birational isomorphism with $\GG$.

\subsection*{6.B.\;Group action on the set of cross-sections}  Let ${\rm Mor}(\GG, G)$ be the group of
morphisms $\GG\to G$. If $S$ is a cross-section in $G$
and $\gamma\in {\rm Mor}(\GG, G)$, then\vskip 1mm
$$
\gamma(S):=\{\gamma(s)s\gamma(s)^{-1}\mid s\in S\}
$$
\vskip 2mm
\noindent is a cross-section in $G$. This defines an
action of ${\rm Mor}(\GG, G)$ on the set of
cross-sections in $G$.
If ${\rm char}\,k=0$, then by
\cite{FM} this action is transitive. If ${\rm
char}\,k>0$, then in the general case
this
is not
true: in Example \ref{exm}, {\sc Steinberg}'s section
and $S$ are not in the same ${\rm Mor}(\GG, G)$-orbit
since, for the former,
the restriction of $\pG$ is separable \cite[Theorem
1.5]{St1}, but, for the latter, it is not.

\subsection*{6.C.\;Lifting \boldmath$T$-action} \label{lift} By
%%Theorem \ref{toric}
Corollary \ref{toricstr} there is an
action of $T$ on $T/W$ determining a structure of a
toric variety. This action cannot be lifted to $T$
making $\pi^{}_{W, T}\colon T\to T/W$ equivariant. This
follows from the fact that the automorphism group of the
underlying variety of $\,T$ is ${\bf GL}_r({\bf
Z})\ltimes T$.

\subsection*{6.D.\;Image of a rational cross-section in \boldmath$G$ under $\pG$} Assume that $\tau$ is not bijective (for ${\rm char}\,k=0$, this means that $G$ is not simply connected). Let $S$ be a rational cross-section in $G$ such that $\varphi:=\pG|_S\colon S\to \GG$ is a birational isomorphism ($S$ exists by Corollary \ref{rtcrsct}). Let $D$ be the closure of the complement of $\,\pG(S)$ in $\GG$.

The following shows that $D$ cannot be
``too small''.
\begin{theorem} $\cod_{\GG}D=1$.
\end{theorem}
\begin{proof} Assume the contrary. Take a function $f\in k[S]$. Since $\varphi$ is a birational iso\-mor\-phism, $f=\varphi^*(h)$ for some function $h\in k(\GG)$. As $\GG$ is normal,
$h$ is regular at every point of $(\GG)\setminus D$, see \cite[Sect.\;2, Lemma]{P1}. Using again that $\GG$ is normal, we then deduce from $\cod_{\GG}D>1$ that $h\in k[\GG]$. As $G$ and $\GG$ are affine and $S$ is closed in $G$, this shows that $\varphi$ is an isomorphism.
 Hence $S$ is a (global) cross-section in $G$. As $\tau$ is not bijective, the latter contradicts Theorem \ref{main}(i).
 \quad $\square$
\renewcommand{\qed}{}\end{proof}

\subsection*{6.E.\;Questions} Given Theorem
%%\ref{rsrcs} and Corollary
\ref{ratcs}, it would be interesting to construct
explicitly an example of a $W$-equivariant rational map
$T\dashrightarrow G/T$.
%% for central~$\tau $.

---\;\;Is there such a map defined on $T_0$?

---\;\;Is there a rational section of $\pG$ defined on
$(\GG)_0$?


\begin{thebibliography}{XXXXX}

\bibitem[Ada]{A} {\rm  J.\;F.\;Adams}, {\it Lectures on Lie Groups}, Benjamin, New York, 1969. Russian transl.: {\rus D{zh}.\;Adams}, {\rusit Lekcii po gruppam Li}, {\rus Nauka, M., 1979.}

    \bibitem[Ben]{Ben} {\rm  D.\;J.\;Benson}, {\it Polynomial Invariants of Finite Groups}, London Mathematical Society Lecture Note Series, Vol. 190, Cambridge University Press, Cambridge, 1993.

\bibitem[Bor]{Bor} {\rm  A.\;Borel}, \emph{Linear Algebraic
Groups},
 2nd enlarged ed., Graduate Texts in Mathematics,
  Vol. 126, Sprin\-ger-Verlag, 1991.

 \bibitem[BT]{BT} {\rm  A.\;Borel, J.\;Tits}, \emph{Compl\'ements
 \`a l'article {\rm``}\hskip -.3mm Groupes r\'eductifs\hskip .2mm{\rm''}}, Publ. math. IHES {\bf 41} (1972), 253--276.

\bibitem[Bou${}_1$]{Bour} {\rm N.\;Bourbaki},
{\it Alg\`ebre Commutative}, Chap. V, VI, Hermann,
Paris, 1964. 
%%Russian transl.: {\rus N.\;Burbaki}, {\rusit Kommutativnaya algebra}, {\rus Mir, %%M., 1971}.

\bibitem[Bou${}_2$]{Bou} {\rm N.\;Bourbaki},
{\it Groupes et Alg\`ebres de Lie}, Chap. IV, V, VI,
Hermann, Paris, 1968. 
%%Russian transl.: {\rus N.\;Burbaki}, {\rusit Gruppy i algebry Li. Gruppy Kokste\-ra %%i sistemy Ti{t}{s}a. Gruppy, poro{zh}dennye otra{zh}eni{ya}mi. Sistemy %%kor\-ne{i0}}, {\rus Mir, M., 1972}.

\bibitem[CTKPR]{CTKPR}
{\rm J.-L.\;Colliot-Th\'el\`ene, B.\;Kunyavski\v \i, V.\;L.\;Popov, Z.\;Reichstein}, {\it Is the function field of a
reductive Lie algebra purely transcendental over the
field of invariants for the adjoint action}?,
%%available
%%at
{\tt arXiv:0901.4358v1} (27 January, 2009).
%% and LAGRS preprint
%%no.\;320 (2009).



\bibitem[FM]{FM}
{\rm R.\;Friedman, J.\;W.\;Morgan}, {\it Automorphism
sheaves, spectral covers, and the Kostant and Steinberg
sections}, in: {\it Vector Bundles and Representation
Theory} (Columbia, MO, 2002), Contemp. Math., Vol. 322,
Amer. Math. Soc., Providence, RI, 2003, pp. 217--244.

\bibitem[Ful]{F} {\rm W.\;Fulton}, {\it Introduction to Toric Varieties}, Princeton University Press, Prin\-ce\-ton, New Jersey, 1993.

\bibitem[Gro${}_1$]{G} {\rm A.\;Grothendieck},
{\it Compl\'ements de g\'eom\'etrie alg\'ebrique.
Espaces de transformations}, in: {\it S\'eminaire C.
Chevalley, $1956$--$1958$. Classification de groupes de
Lie alg\'ebriques}, Vol.\,1, Expos\'e no. 5, Secr. math.
ENS, Paris, 1958.

\bibitem[Gro${}_2$]{G2} {\rm A.\;Grothendieck},
{\it EGA} I, Publ. Math. IHES {\bf 4} (1960), 5--228.

\bibitem[Gro${}_3$]{G3} {\rm A.\;Grothendieck et al.}, {\it Rev\^etements Etales et Groupe Fondamental}, Lecture Notes in Mathematics,
Vol. 224, Springer-Verlag, Berlin, 1991.

\bibitem[GS]{GS} {\it Grothendieck--Serre Correspondence}, Bilingual Edition, {\rm P.\;Colmez}, {\rm J.-P.\;Serre}, eds., American Mathematical Society, Soci\'et\'e Math\'ematique de France, 2004.

\bibitem[HR]{HR} {\it G.\;H.\;Hardy, S.\;Ramanujan}, {\it Asymptotic formulae in combinatory analysis}, Proc. London Math. Soc. {\bf 17} (1918), 75--115.


\bibitem[Har]{Har} {\rm J.\;Harris}, {\it Algebraic Geometry. A First Course}, Graduate Texts in Mathematics, Vol. 133, Springer-Verlag, New York, 1995.
    %%Russian transl.: {\rus D{zh}.\;Harris}, {\rus Algebrai{ch}eska{ya} %%geometri{ya}. Na{ch}al{p1}ny{i0} kurs}, {\rus MCNMO, M., 2006}.

    \bibitem[Hum$_{1}$]{Hum1} {\rm J.\;E.\;Humphreys},
{\it Introduction to Lie Algebras and Representation Theory}, Graduate Texts in Mathematics, Vol. 9, Springer-Verlag, New York, 1972.
 %%Russian transl.: {\rus D{zh}.\;Hamfris}, {\rusit Vvedenie v teori{yu} algebr Li i ih %%predstavleni{i0}}, {\rus MCNMO, M., 2003}.

\bibitem[Hum$_{2}$]{Hum} {\rm J.\;E.\;Humphreys},
{\it Linear Algebraic Groups}, Graduate Texts in Mathematics, Vol. 21, Springer-Verlag, New
York, 1975. 
%%Russian transl.: {\rus D{zh}.\;Hamfris}, {\rusit Line{i0}nye algebrai{ch}eskie %%gruppy}, {\rus Nauka, M., 1980.}

\bibitem[Hum$_{3}$]{Hum2}
  {\rm J.\;E.\;Humphreys},
 \emph{Conjugacy Classes in Semisimple Algebraic
 Groups}, Mathematical Surveys and Monographs, Vol. 43,
  American Mathematical Society, Providence, RI, 1995.

\bibitem[Hus]{Hus} {\rm D.\;Husemoller}, {\it Fibre Bundles},
McGraw-Hill Book Company, New York, 1966.
%%Russian transl.: {\rus D.\;H{p1}{yu}zmoller}, {\rusit Rassloennye prostranstva}, %%{\rus Mir, M., 1970.}

\bibitem[Kac]{Kac} {\rm V.\;G.\;Kac}, {\it Root systems, representations of quivers and invariant theory}, in: {\it Invariant Theory}, Proceedings, Montecatini 1982,
    Lecture Notes in Mathematics, Vol. 996, Springer-Verlag, Berlin, 1983, pp.\;74--108.


\bibitem[Kos]{K} {\rm B.\;Kostant},
{\it Lie group representations on polynomial rings},
{\rm Amer. J. Math.} {\bf 85} (1963), 327--404.

\bibitem[Lor${}_1$]{L0} {\rm M.\;Lorenz}, {\it Multiplicative invariants and semigroup algebras}, Algebras and Representation Theory {\bf 4} (2001), 293--304.

\bibitem[Lor${}_2$]{L} {\rm M.\;Lorenz}, {\it
Multiplicative Invariant Theory}, Encyclopaedia of
Mathematical Sciences, Vol. 135, Subseries {\it
Invariant Theory and Algebraic Transformation Groups},
Vol. VI, Springer, Berlin, 2005.

\bibitem[MF]{MF} {\rm D.\;Mumford, J.\;Fogarty}, {\it Geometric Invariant Theory}, Ergebnisse der Mathematik und ihrer Grenzgebiete, Band 34, Springer-Verlag, Berlin, 1982.




\bibitem[Oda]{O} {\rm T.\;Oda}, {\it Convex Bodies and
Algebraic Geometry},
Ergebnisse der Mathematik und ihrer Grenzgebiete 3.
Folge, Band 15, Springer-Verlag, Berlin, 1988.

\bibitem[Pop${}_1$]{P1} {\rm V.\;L.\;Popov}, {\it On the {\rm``}Lemma of Seshadri\,{\rm''}}, in:\;{\it Lie Groups, their Discrete Subgroups, and Invariant Theory}, Advances in Soviet Mathematics, Vol. 8, Amer. Math. Soc., Providence, RI, 1992, 167--172.

\bibitem[Pop${}_2$]{P2} {\rm V.\;L.\;Popov}, {\it Letter to A.\,Premet}, July 5, 2009.


\bibitem[Rich${}_1$]{R1} {\rm R.\;W.\;Richardson},
{\it The conjugating representation of a semisimple
group}, Invent. Math. {\bf 54} (1979), 229--245.

\bibitem[Rich${}_2$]{R2} {\rm R.\;W.\;Richardson},
{\it Orbits, invariants, and representations associated
to involutions of reductive groups}, Invent. Math. {\bf
66} (1982), 287--312.





\bibitem[Ser${}_1$]{Se2} {\rm J.-P.\;Serre},
{\it Groupes de Grothendieck des sch\'emas en groupes
r\'eductifs d\'eploy\'es}, Publ. Math. IHES {\bf 34}
(1968), 37--52.

\bibitem[Ser${}_2$]{Se3} {\rm J.-P.\;Serre}, {\it Groupes finis d'automorphismes d'anneaux locaux r\'eguliers}, in:
    {\it Colloque d'Alg\`ebre}, Secr\'etariat mathe\'ematique, Paris, 1968, pp.\;8-01--8-11.


\bibitem[Slo]{Sl} {\rm P.\;Slodowy}, {\it Simple Singularities and Simple Algebraic Groups},
Lecture Notes in Mathematics, Vol. 815, Springer-Verlag,
Berlin, 1980.

\bibitem[Spr]{Sp} {\rm T.\;A.\;Springer}, {\it Linear Algebraic Groups}, 2nd ed., Birkh\"auser, Boston, 1998.


\bibitem[Ste${}_1$]{St1} {\rm R.\;Steinberg}, {\it Regular elements of semi-simple
algebraic groups}, Publ. Math. IHES {\bf 25} (1965),
49--80.

\bibitem[Ste${}_2$]{St2} {\rm R.\;Steinberg}, {\it Lectures on Chevalley Groups}, Yale
University, New Haven, Conn., 1968.



\bibitem[Ste${}_3$]{St4} {\rm R.\;Steinberg}, {\it On a theorem of Pittie}, Topology {\bf
14} (1975), 173--177.

\bibitem[Stu]{Stu} {\rm B.\;Sturmfels},
{\it Gr\"obner Bases and Convex Polytopes}, University
Lecture Series, Vol. 8, American Mathematical Society,
Providence, Rhode Island, 1996.
\end{thebibliography}
\end{document}